\documentclass{amsart}
\usepackage{color}
\usepackage{amsmath}
\usepackage{amssymb}
\usepackage{amsthm}
\usepackage{amscd}
\usepackage{eucal} 
\usepackage{enumitem}
\usepackage{tikz}
\usepackage{hyperref}
\usepackage[utf8]{inputenc}
\usepackage[all]{xy}
\hypersetup{colorlinks=true}

\theoremstyle{plain}
\newtheorem{theorem}{Theorem}[section]

\newtheorem{lemma}[theorem]{Lemma}
\newtheorem{proposition}[theorem]{Proposition}
\newtheorem{corollary}[theorem]{Corollary}

\theoremstyle{definition}
\newtheorem{definition}[theorem]{Definition}
\newtheorem{remark}[theorem]{Remark}
\newtheorem{example}[theorem]{Example}

\numberwithin{equation}{section}

\newcommand\fantome[1]{}

\def\bL{\mathbb L}
\def\bN{\mathbb N}

\def\bC{\mathbb C}

\def\bT{\mathbb T}

\def\cF{\mathcal F}

\def\calM{\mathcal M}

\def\Fq{\mathbb F_q}
\def\lra{\longrightarrow}

\newcommand{\LL}{\mathbb{L}}

\newcommand{\iso}{\overset{\sim}{\longrightarrow}}

\newcommand{\fLis}{\mathfrak{Li}^{*}}

\DeclareMathOperator{\ev}{ev}
\DeclareMathOperator{\Log}{Log}

\DeclareMathOperator{\Lie}{Lie}
\DeclareMathOperator{\Li}{Li}
\DeclareMathOperator{\fLi}{\mathfrak{Li}}
\DeclareMathOperator{\fL}{\mathfrak{L}}

\newcommand{\power}[2]{{#1 [[ #2 ]]}}

\newcommand{\F}{\mathbb{F}}
\newcommand{\C}{\mathbb{C}}

\newcommand{\bu}{\mathbf{u}}
\newcommand{\bff}{\mathbf{f}}
\newcommand{\bm}{\mathbf{m}}

\newcommand{\bv}{\mathbf{v}}

\newcommand{\bz}{\mathbf{z}}
\newcommand{\by}{\mathbf{y}}
\newcommand{\bx}{\mathbf{x}}

\newcommand{\bw}{\mathbf{w}}

\newcommand{\fs}{\mathfrak{s}}
\newcommand{\fQ}{\mathfrak{Q}}

\newcommand{\inv}{\ensuremath ^{-1}}
\newcommand{\isom}{\ensuremath \cong}

\newcommand{\TT}{\mathbb{T}}
\newcommand{\N}{\ensuremath \mathbb{N}}

\DeclareMathOperator{\Exp}{Exp}

\DeclareMathOperator{\Mat}{Mat}

\newcommand{\twist}{^{(1)}}

\newcommand{\twistinv}{^{(-1)}}
\newcommand{\twisti}{^{(i)}}
\newcommand{\twistk}[1]{^{(#1)}}

\definecolor{ForestGreen}{rgb}{0.0, 0.5, 0.0}

\author{Nathan Green}
\address{Department of Mathematics\\
University of California, San Diego (UCSD)\\
9500 Gilman Drive 0112\\
La Jolla, CA  92093-0112\\
United States of America (USA)\\}
\email{n2green@ucsd.edu}

\author{Tuan Ngo Dac}
\address{
CNRS - Université Claude Bernard Lyon 1, 
Institut Camille Jordan, UMR 5208,
43 boulevard du 11 novembre 1918,
69622 Villeurbanne Cedex, France
}
\email{ngodac@math.univ-lyon1.fr}

\title[Log-algebraicity and MZV's]{On log-algebraic identities for Anderson $t$-modules and characteristic $p$ multiple zeta values}

\begin{document}

\begin{abstract}
Based on the notion of Stark units we present a new approach that obtains refinements of log-algebraic identities for Anderson $t$-modules. As a consequence we establish a generalization of Chang's theorem on logarithmic interpretations for special characteristic $p$ multiple zeta values (MZV's) and recover many earlier results in this direction. Further, we devise a direct and conceptual way to get logarithmic interpretations for both MZV's and $\nu$-adic MZV's. This generalizes completely the work of Anderson and Thakur for Carlitz zeta values.
\end{abstract}



\date{\today}

\maketitle

\tableofcontents


\section{Introduction}

\subsection{Background} ${}$\par



The power-series $\sum_{n \geq 1} \frac{z^n}{n}$ is log-algebraic:
	\[ \sum_{n \geq 1} \frac{z^n}{n}=-\log(1-z). \]
This identity allows one to obtain the value of a Dirichlet $L$-series at $s=1$ as an algebraic linear combination of logarithms of circular units. 


By a well-known analogy between the arithmetic of number fields and that of global function fields, conceived of in the 1930s by Carlitz, we now switch to the function field setting. We briefly recount some of the many advances which have been made in function field arithmetic.  In particular we will focus on the study of special values of Goss $L$-functions and their generalizations, like Thakur's characteristic $p$ multiple zeta values (MZV's for short).  Especially, we wish to highlight the reliance many of these results have on log-algebraic identities. 

We let $A=\mathbb F_q[\theta]$ with $\theta$ an indeterminate over a finite field $\mathbb F_q$. In the 1930's Carlitz \cite{Car35} introduced the Carlitz zeta values $\zeta_A(n)$ for $n \in \N$, which are analogues of positive special values of the Riemann zeta function, $\zeta(n)$. He then related the zeta value $\zeta_A(1)$ to the so-called Carlitz module $C$. One of his fundamental theorems gave a log-algebraic identity
	\[ \exp_C(\zeta_A(1))=1 \]
where $\exp_C$ is the exponential series attached to the Carlitz module. We mention that Goss \cite{Gos79} introduced a new type of $L$-functions in the arithmetic of function fields over finite fields and showed that Carlitz zeta values can be realized as special values of such $L$-functions (see \cite{Gos96}, Chapter 8).

In the 1970's Drinfeld \cite{Dri74,Dri77} made a breakthrough and defined Drinfeld modules even for a more general ring $A$. It turned out that the Carlitz module is the simplest example of a Drinfeld module. Several years later Anderson \cite{And86} developed the theory of $t$-modules which are higher dimensional generalizations of Drinfeld modules. 

Since the introduction of $t$-modules, several additional log-algebraic identities for Anderson $t$-modules have been discovered. The theory began with the seminal paper of Anderson and Thakur \cite{AT90} where they proved log-algebraic identities for tensor powers $C^{\otimes n}$ $(n \in \N)$ of the Carlitz module. The latter result implies logarithmic interpretations for Carlitz zeta values $\zeta_A(n)$ at positive integers $n$ generalizing the aforementioned result of Carlitz. Combining the above result with his transcendence theory, Yu \cite{Yu91} proved that $\zeta_A(n)$ is transcendental for all positive integers $n$. Based on the criteria for linear and algebraic independence developed by Jing Yu \cite{Yu97}, Anderson-Brownawell-Papanikolas \cite{ABP04} and Papanikolas \cite{Pap08}, Chang and Yu \cite{CY07,Yu97} determined all algebraic relations among the Carlitz zeta values. These results are very striking when compared to the extremely limited knowledge we have about the transcendence of odd Riemann zeta values in the classical setting.

In recent years various works have revealed the importance of log-algebraicity on Anderson $t$-modules in function field arithmetic. On the one hand, following the pioneering work of Anderson \cite{And96} in which he introduced the analogue of cyclotomic units for the Carlitz module, Anglès, Tavares Ribeiro and the second author have developed the theory of Stark units for Anderson modules which turns out to be a powerful tool for investigating log-algebraicity. Roughly speaking, they are units in the sense of Taelman \cite{Tae10,Tae12a} coming from the canonical deformation of Drinfeld modules in Tate algebras in the sense of Pellarin \cite{Pel12}. Note that the concept of Stark units appeared implicitly in \cite{APTR16,APTR18}. The notion was formalized in \cite{ATR17} for Drinfeld modules over $\F_q[\theta]$ and then further developed in more general settings in \cite{ANDTR17a,ANDTR20,ANDTR20a}. Recently, combining Stark units and the class formula à la Taelman, Anglès, Tavares Ribeiro and the second author \cite{ANDTR20a} obtained various log-algebraicity results for tensor powers of the Carlitz module, generalizing the work of Anderson-Thakur \cite{AT90} and recovering that of Papanikolas \cite{Pap}. On the other hand, log-algebraicity has been successfully applied to the study of Goss's zeta values and Thakur's characteristic $p$ multiple zeta values (MZV's). For example, using log-algebraicity Chang \cite{Cha16} completely determined linear relations among depth-two MZV's, the authors \cite{GND20} generalized the work of Chang and Yu \cite{CY07} by completely determining algebraic relations among Goss's zeta values on function fields of elliptic curves, and Chang and Mishiba \cite{CM19b} proved a conjecture of Furusho concerning MZV's and their $\nu$-adic variants over function fields. 

In the present paper, based on the notion of Stark units, we introduce a new approach to obtain refinements of log-algebraic identities for Anderson $t$-modules. One of the main benefits of our approach is that it provides a concise general theory on the existence of log-algebraic identities, and thus it gives a unifying framework to many such previous results which have been proven in a somewhat ad-hoc fashion.  To demonstrate the unification our new techniques allow, we use them to recover many previously known results in a straightforward way, and in some cases our techniques even lead to stronger results.  We also apply our techniques to prove new formulas relating to characteristic $p$ MZV's.

For applications of our new techniques, we first investigate the dual $t$-motives introduced by Anderson and Thakur \cite{AT09} and developed further by Chang, Mishiba, Papanikolas, Yu and the first author (see \cite{Cha14,Cha16,CGM19,CM19,CM19b,CPY19}). Our main result yields log-algebraic identities for the $t$-modules attached to these dual $t$-motives. Next we obtain a generalization of one of the main theorems of Chang in \cite{Cha16} where he presented very simple and elegant logarithmic interpretations for special cases of MZV's. Along the way, we clarify connections between these $t$-modules and MZV's and recover many results in \cite{CGM19,CM19,CM19b,CPY19}. Finally we devise new dual $t$-motives called star motives which provide direct logarithmic interpretations for both MZV's and $\nu$-adic MZV's in the same spirit of the original work of Anderson and Thakur. This generalizes completely the work of Anderson and Thakur \cite{AT90} and answers positively to a problem raised by Chang and Mishiba \cite{CM19b}.  

\subsection{Statement of the main result} ${}$\par

Let us give now more precise statements of our results.

Let $A=\Fq[\theta]$ be the polynomial ring in the variable $\theta$ over a finite field $\Fq$ of $q$ elements of characteristic $p>0$. Let $K=\Fq(\theta)$ be the fraction field of $A$ equipped the rational point $\infty$. Let $K_\infty$ be the completion of $K$ at $\infty$ and $\C_\infty$ be the completion of a fixed algebraic closure $\overline K$ of $K$ at $\infty$. Letting $t$ be another independent variable, we denote by $\TT$ the Tate algebra in the variable $t$ with coefficients in $\C_\infty$ and by $\bL$ the fraction field of $\bT$.

Let $\overline K[\tau]$ (resp. $\overline K[\sigma]$) denote the non-commutative skew-polynomial ring with coefficients in $\overline K$, subject to the relation for $c \in \overline K$,
\[\tau c = c^q\tau \quad \text{(resp. $\sigma c = c^{1/q} \sigma)$}.\]
We define Frobenius twisting on $\overline K[t]$ by setting for $g = \sum_j c_j t^j \in  \overline K[t]$, 
\begin{equation*}
g\twisti = \sum_j c_j^{q^i} t^j.
\end{equation*}
We extend twisting to matrices in $\Mat_{i\times j}(\overline K[t])$ by twisting coordinatewise.  

We will work with effective dual $t$-motives and Anderson $t$-modules introduced by Anderson (see \cite{And86,BP20,HJ20}). In what follows, we let $\calM'$ denote an effective dual $t$-motive in the sense of \cite[\S 4]{HJ20}, which is a $\overline K[t,\sigma]$-module that is free and finitely generated over $\overline K[t]$ such that for $\ell\gg 0$ we have $(t-\theta)^\ell(\calM'/\sigma \calM') = \{0\}$. Letting $\bm=\{m_1,\ldots,m_r\}$ be a $\overline K[t]$-basis of $\calM'$, then there exists a unique matrix $\Phi' \in \Mat_r(\overline K[t]) \cap \text{GL}_r(\overline K(t))$ such that 
	\[ \sigma \bm=\Phi' \bm. \]
We suppose further that $\calM'$ is free and finitely generated over $\overline K[\sigma]$ and that $\calM'$ is uniformizable or rigid analytically trivial, which means that there exists a matrix $\Psi' \in \text{GL}_r(\bL)$ satisfying $\Psi'^{(-1)}=\Phi' \Psi'$.
	
Anderson associated to $\calM'$ an Anderson $t$-module $E'$ defined over $\overline K$ (see \cite[\S 5.2]{HJ20}). This is an $\Fq$-algebra homomorphism $E':A \lra \Mat_d(\overline K)[\tau]$ for some $d \in \N$ (called the dimension of $E'$) such that for all $a \in A$, if we write 
	\[ E'_a = d[a]+E'_{a,1} \tau+\ldots, \] 
then we have $(d[a]-aI_d)^d=0$.  Note that for any $\overline K$-algebra $B$, we can define two $A$-module structures on $B^d$: the first one is denoted by $E(B)$ where $A$ acts on $B^d$ via $E$, and the second one is denoted by $\Lie_E(B)$ where $A$ acts on $B^d$ via $d[\cdot]$.

The association of $E'$ with $\calM'$ comes with two canonical maps (see \eqref{D:deltaextendedmaps})
\begin{align*}
\delta_0: \calM' \to \Mat_{d \times 1}(\overline K), \quad \delta_1: \calM' \to \Mat_{d \times 1}(\overline K),
\end{align*}
which extend to $\calM' \otimes_{\overline K[t]} \bT \simeq \Mat_{r \times 1}(\TT)$ in the natural way.
One can show that there exists a unique exponential series $\Exp _{E'} \in I_d+\tau \Mat_d(\overline K)[[\tau ]]$ associated to $E'$ such that
	\[ \Exp_{E'} d[a] = E'_a \Exp_{E'}, \quad a\in A. \]
The logarithm function $\Log_{E'}$ is then defined as the formal power series inverse of $\Exp_{E'}$. We note that as functions on $\C_\infty^d$ the function $\Exp_{E'}$ is everywhere convergent, whereas $\Log_{E'}$ has some finite radius of convergence.

Let $\calM \in \textrm{Ext}^1_\cF(\mathbf{1},\calM')$ be the effective dual $t$-motive given by the matrix
\begin{align*}
\Phi=\begin{pmatrix}
\Phi' & 0  \\
\bff & 1 
\end{pmatrix}, \quad \text{with } \bff=(f_1,\dots,f_r)\in \Mat_{1\times r}(\overline K[t]).
\end{align*}
Let $\Psi$ be a rigid analytic trivialization such that we can write
\begin{align*}
\Psi=\begin{pmatrix}
\Psi' & 0  \\
\Psi_{r+1} & 1 
\end{pmatrix} \in \text{GL}_{r+1}(\bL), \quad \text{with } \Psi_{r+1} \in \Mat_{1\times r}(\bL),
\end{align*}
and 
\begin{align*} 
\Upsilon:=\Psi\inv=\begin{pmatrix}
\Upsilon' & 0  \\
\Upsilon_{r+1} & 1 
\end{pmatrix} \in \Mat_{r+1}(\bT), \quad \text{with } \Upsilon_{r+1} \in \Mat_{1\times r}(\TT).
\end{align*}
Note that by \cite[Proposition 3.3.9]{Pap08} there exists a polynomial $F \in \Fq[t]$ such that $F \Psi \in \Mat_{r+1}(\bT)$.

Inspired by \cite{CPY19} we construct a point $\delta_1  (\bff^\top) = \bv_{\calM} \in E'(\overline K)$ (see \eqref{D:vMdef}) associated to the extension $\calM \in \textrm{Ext}^1_\cF(\mathbf{1},\calM')$.

We now recall the notion of units and Stark units. We mention that the former was introduced by Taelman in \cite{Tae10} and the latter has been introduced and developed in \cite{ANDTR17a,ANDTR20,ANDTR20a,ATR17} following the pioneering work of Anderson \cite{And96} in which he introduced the analogue of cyclotomic units for the Carlitz module. The notion of Stark units turns out to be a powerful tool for arithmetic applications including log-algebraic identities and Taelman's class formula. Let $z$ be an indeterminate with $\tau z=z \tau$ and let $\bT_z(\C_\infty)$ be the Tate algebra in the variable $z$ with coefficients in $\C_\infty$. We define the canonical $z$-deformation $\widetilde E'$, which is called the $t$-module defined over a Tate algebra in the sense of Pellarin \cite{Pel12} (see also \cite{APTR16,APTR18}). It is the homomorphism of $\Fq[z]$-algebras $\widetilde E': A[z] \to \Mat_d(\overline K[z])[\tau]$ such that 
	\[ \widetilde E'_a =\sum_{k\geq 0} E'_{a,k} z^k \tau ^k, \quad a \in A. \]
Then there exists a unique series $\Exp_{\widetilde E'}\in I_d+\tau \Mat_d(\overline K[z])\{\{\tau \}\}$ such that
	\[ \Exp_{\widetilde E'} d[a] = \widetilde E'_a \Exp_{\widetilde E'}, \quad a \in A. \]
One can show that if we write $\Exp_{E'}=\sum_{i\geq 0} Q_i \tau ^i$, then $\Exp_{\widetilde E'}=\sum_{i\geq 0} Q_i z^i \tau ^i$. Thus $\Exp_{\widetilde E'}$ converges on $\Lie_{\widetilde E'}(\bT_z(\C_\infty))$ and induces  a homomorphism of $A[z]$-modules 
	\[ \Exp_{\widetilde E'}:\Lie_{\widetilde E'}(\bT_z(\C_\infty))\to \widetilde E'(\bT_z(\C_\infty)). \]
We denote by $\Log_{\widetilde E'} \in I_d+\tau \Mat_d(\overline K[z])[[\tau ]]$ the inverse of $\Exp_{\widetilde E'}$. 

\begin{definition}[Definitions \ref{definition: log_St} and \ref{definition: split log}]
Letting $\bz \in \C_\infty^d$, we put $\mathbf{Z}:=\Exp_{E'}(\bz)$. 

1) Following Taelman \cite{Tae10} we say that $\bz$ is {\it a unit for $E'$} if $\mathbf{Z}=\Exp_{E'}(\bz) \in \overline K^d$. The latter equality is also called {\it a log-algebraic identity for $E'$}.

2) We say that $\bz$ is {\it a Stark unit for $E'$} if we can write $\bz=\bx \big|_{z=1}$ for some $\bx \in \Lie_{\widetilde E'}(\bT_z(\C_\infty))$ satisfying $\Exp_{\widetilde E'}(\bx) \in \overline K[z]^d$.  We also say that we have {\it a Stark logarithmic identity for $\bz$}, and we write
	\[ \Log_{E'}^{St}(\mathbf{Z}) = \bz.\]
	
3) Suppose that there exists some finite collection of pairs $\{(a_i,\bu_i)\}\subset A\times \C_\infty^d$ where each $\bu_i$ is in the domain of convergence of $\Log_{E'}$, such that
\begin{align*}
\bz &= \sum_{i} d[a_i] \Log_{E'}(\bu_i),
\end{align*}
we will say that we have {\it a split-logarithmic identity for $\bz$}, and we write
\[\Log_{E'}^!(\mathbf{Z}) = \bz.\]
\end{definition}

We mention that Stark logarithmic identities and split-logarithmic identities are much stronger than log-algebraic identities and could be viewed as refinements thereof. We emphasize that compared to log-algebraic identities (resp. units), Stark logarithmic identities (resp. Stark units) allows one to bypass convergence issues arising from logarithmic series and to ``deal directly" with logarithms. In recent years the notion of Stark units has been successfully applied to achieve important results related to special values of the Goss $L$-functions, characteristic $p$ multiple zeta values, Anderson's log-algebraicity identities, Taelman's units, and Drinfeld modular forms in Tate algebras (see \cite{ANDTR17a,ANDTR20,ANDTR20a,APTR16,APTR18,ATR17,Gre17b,GND20,GP18}).

We also comment that split-logarithmic identities are stronger than Stark logarithmic identities. They are common when one discovers a log-algebraic identity of the form $\Exp_{E'}(\bz) = \mathbf{Z}$, but $\mathbf{Z}$ is not inside the domain of convergence of $\Log_{E'}$. Thakur \cite{Tha92} suggested that one can decompose $\mathbf{Z}$ into a sum of terms $E'_{a_i}\bu_i$, such that each $\bu_i$ is inside the domain of convergence of $\Log_{E'}$.  Such a decomposition is the motivation for the above definition of split-logarithmic identity.  This is the case in the celebrated log-algebraicity theorem of Anderson and Thakur for tensor powers of the Carlitz module \cite[Theorem 3.8.3]{AT90} (see also \cite[\S II]{Tha92}).  

We are now ready to state the main result of our paper (see \S \ref{sec: main theorem} for more details).

\begin{theorem}[Theorem \ref{theorem: main theorem}] \label{theorem: Intro main theorem}
Let $\Psi$, $\Upsilon$ and $f$ be defined as above.  Let $F \in \Fq[t]$ such that $F \Psi \in \Mat_{r+1}(\bT)$. We suppose that $F\Psi_{r+1}^{(k)}$ converges to $0$ as a vector of functions in $\Mat_{1\times r}(\TT)$. 

Then the point $\delta_0(\bff^\top-\Upsilon_{r+1}^\top)$ is a Stark unit for $E'$ and we have a Stark logarithmic identity
	\[ \Log^{\text{St}}_{E'}(\bv_{\calM})=\delta_0(\bff^\top-\Upsilon_{r+1}^\top). \]
Further, if the point $\bv_{\calM}$ satisfies some mild convergence conditions, then we have a split-logarithmic identity 
	\[ \Log^!_{E'}(\bv_{\calM})=\delta_0(\bff^\top-\Upsilon_{r+1}^\top). \]
\end{theorem}

We sketch now main ideas of the proof of Theorem \ref{theorem: Intro main theorem}. 
\begin{enumerate}
\item We explicitly compute the coefficient matrices of the logarithm series $\Log_{E'}$ of $E'$. Our method is based on a joint work of the second author with Anglès and Tavares Ribeiro \cite{ANDTR20a} and is different from the approach of Anderson and Thakur \cite{AT90}. It also differs from the logarithm computations of the first author, Chang and Mishiba in \cite{Gre17b,CGM19}.

\item Then we consider the canonical $z$-deformation $\widetilde E'$ of the Anderson $t$-module $E'$ and also the canonical $z$-deformation $\bv_{\calM,z}$ of $\bv_{\calM}$. Using Step (1) we compute the value $\Log_{\widetilde{E}'}(\bv_{\calM,z})$ as a formal series in $z$.

\item Using the hypothesis on $\Psi_{r+1}$, we show that the formal series $\Log_{\widetilde{E}'}(\bv_{\calM,z})$ belongs to the Tate algebra $\Lie_{\widetilde{E}'}(\bT_z(\C_\infty))$ in the variable $z$ and with coefficients in $\C_\infty$. Hence we obtain the desired result.
\end{enumerate}

\subsection{Applications of the main result} ${}$\par

We present several applications of our main result to Carlitz zeta values and characteristic $p$ multiple zeta values.  We briefly recall the definitions of these values. In \cite{Car35} Carlitz introduced the Carlitz zeta values $\zeta_A(n)$ $(n \in \N)$ given by 
	\[ \zeta_A(n) := \sum_{a \in A_+} \frac{1}{a^n} \in K_\infty \]
which are analogues of classical special zeta values in the function field setting. Here $A_+$ denotes the set of monic polynomials in $A$. For any tuple of positive integers $\mathfrak s=(s_1,\ldots,s_r) \in \N^r$, Thakur \cite{Tha04} defined the characteristic $p$ multiple zeta value (MZV for short) $\zeta_A(\fs)$ or $\zeta_A(s_1,\ldots,s_r)$ by
\begin{equation*}
\zeta_A(\fs):=\sum \frac{1}{a_1^{s_1} \ldots a_r^{s_r}} \in K_\infty
\end{equation*}
where the sum runs through the set of tuples $(a_1,\ldots,a_r) \in A_+^r$ with $\deg a_1>\ldots>\deg a_r$. We call $r$ the depth of $\zeta_A(\fs)$. We note that Carlitz zeta values are exactly depth one MZV's. 

In \cite{AT09}, for $\fs=(s_1,\ldots,s_r) \in \N^r$ as above, Anderson and Thakur used Anderson-Thakur polynomials to construct an effective dual $t$-motive which is rigid analytically trivial such that $\zeta_A(\fs)$ appears in the entries of the attached rigid analytic trivialization. It has been generalized to the so-called Anderson-Thakur (effective) dual $t$-motives $\calM'_{\fs,\fQ}$ indexed by more general tuples of polynomials $\fQ=(Q_1,\dots,Q_r) \in \overline K[t]^r$ and notably developed by Chang \cite{Cha14} and Chang, Papanikolas and Yu \cite{CPY19} in recent years (see also \cite{Cha16,CGM19,CM19,CM19b}).

Section \ref{S:Log for And Thak Models} is devoted to the applications of the main result (see Theorem \ref{theorem: Intro main theorem}) to the setting of the Anderson-Thakur dual $t$-motives. Inspired by \cite{CPY19} we define the $t$-module $E'_{\fs,\fQ}$ associated to the dual $t$-motive $\calM'_{\fs,\fQ}$ and the associated special point $\bv_{\fs,\fQ} \in E'_{\fs,\fQ}(\overline K)$. We then establish a split-logarithmic identity for $E'_{\fs,\fQ}$:

\begin{theorem}[Theorem \ref{theorem: AT model}] \label{theorem: AT model intro}
We have a split-logarithmic identity
\begin{align*}
\Log_{E'}^!(\bv_{\fs,\fQ})
=\delta_0 \begin{pmatrix}
(-1)^{r-1} \mathfrak L^*(s_r,\dots,s_1) \Omega^{-(s_1+\dots+s_r)} \\
(-1)^{r-2} \mathfrak L^*(s_r,\dots,s_2) \Omega^{-(s_2+\dots+s_r)} \\
\vdots \\
\mathfrak L^*(s_r) \Omega^{-s_r} 
\end{pmatrix},
\end{align*}
where the series $\mathfrak L^*$ are defined in \eqref{eq: series L star} following Chang \cite{Cha14}.
\end{theorem}

When we specialize $\fQ$ to Anderson-Thakur polynomials (see \S \ref{sec: CPY}), the dual $t$-motives are intimately related to MZV's and are well studied in the aforementioned works. In this setting Chang gave very simple and elegant logarithmic interpretations for some special MZV's (see \cite[Theorem 1.4.1]{Cha16}). However, as Chang and Mishiba \cite{CM20} explained to us, the relations among Chang's theorem and the works of Chang-Papanikolas-Yu \cite{CPY19} and other works \cite{CM19,CM19b} are still mysterious.  The aim of Theorem \ref{theorem: generalization Chang} is twofold. It presents a generalization of Chang's theorem to the general setting, i.e. for any tuple $\fQ$ and also clarifies the connections with the previous works \cite{CM19,CM19b,CPY19}. 

\begin{theorem}[Theorem \ref{theorem: generalization Chang}]
Let $\fs=(s_1,\ldots,s_r) \in \N^r$ with $r \geq 2$. Assume that, for $1 \leq \ell<j\leq r+1$, the values $\mathfrak L(s_\ell,\dots,s_{j-1})(\theta)$ do not vanish (see \eqref{eq: series L general} for a precise definition). We further suppose that $\mathfrak L(s_2,\ldots,s_r)(\theta) \in K$. Then there exist $a_\fs \in A$, an integral point $\mathbf{Z}_\fs \in C^{\otimes (s_1+\ldots+s_r)}(A)$ and a point $\bz_\fs \in \C_\infty^{s_1+\ldots+s_r}$ such that 

1) the last coordinate of $\bz_\fs$ equals $a_\fs \mathfrak L(s_1,\ldots,s_r)(\theta)$,

2) $\Exp_{C^{\otimes (s_1+\ldots+s_r)}}(\bz_\fs)=\mathbf{Z}_\fs$.
\end{theorem} 

Next we apply Theorem \ref{theorem: AT model intro} to the dual $t$-motives studied by Chang, Papanikolas and Yu in \cite{CPY19} in \S \ref{sec: CPY} and to those connected to multiple polylogarithms studied  by Chang, the first author and Mishiba in \cite{CGM19,CM19,CM19b} in \S \ref{sec: CM model}. We recover some earlier results (see Proposition \ref{proposition: torsion} and Theorem \ref{theorem: CM model}) and discover new results, which we state briefly below (see \S \ref{sec: CPY} for precise definitions of $\Gamma_i$ and $\zeta^*_A$).
\begin{theorem}[Theorem \ref{theorem: CPY model}]
For $\fs=(s_1,\ldots,s_r) \in \N^r$, we put $d_\ell:=s_\ell+\dots+s_r$ for $1 \leq \ell \leq r$. Let the polynomials $\fQ$ of Theorem \ref{theorem: AT model intro} be specialized to be Anderson-Thakur polynomials (see \S \ref{sec: CPY}).  Then the $(d_1+\dots+d_\ell)$th coordinate of $\Log_{E'}^!(\bv_{\fs,\fQ})$ of Theorem \ref{theorem: AT model intro} equals $(-1)^{r-\ell} \Gamma_{s_\ell} \dots \Gamma_{s_r} \zeta^*_A(s_r\dots,s_\ell)$.
\end{theorem} 

Section \ref{S:Log Int for MZVs} is devoted to proving new logarithmic interpretations for MZV's and for $\nu$-adic MZV's in the same spirit of the original work of Anderson and Thakur \cite{AT90} for Carlitz zeta values. We note that the entries of $\Upsilon(\theta)=\Psi\inv(\theta)$ attached to the above Anderson-Thakur dual $t$-motives are not MZV's except in the depth one case as in \cite{AT90}. This may explain some of the difficulties encountered when one wishes to extend the work of Anderson and Thakur via this setting (see \cite[Introduction]{CM19b} for a detailed discussion). To bypass this issue, for $\fs=(s_1,\ldots,s_r) \in \N^r$, we devise a  new dual $t$-motive $\calM'^*$ called the star dual $t$-motive whose entries of the associated matrix $\Upsilon^*(\theta)$ naturally contain MZV's. We explicitly construct an Anderson $t$-module $E'^*$ defined over $A$ and an integral point $\bv^*_\fs \in E'^*(A)$. Finally, we apply our main result to obtain the desired logarithmic interpretation for MZV's (see \S \ref{sec:MZV} for related definitions).

\begin{theorem}[Theorem \ref{theorem: star model}]
For $\fs=(s_1,\ldots,s_r) \in \N^r$, we put $d_\ell:=s_\ell+\dots+s_r$ for $1 \leq \ell \leq r$. Then we have
\begin{align*}
\Log_{E'^*}^!(\bv^*_\fs)
=\delta_0 \begin{pmatrix}
- \mathfrak L(s_r,\dots,s_1) \Omega^{-(s_1+\dots+s_r)} \\
- \mathfrak L(s_r,\dots,s_2) \Omega^{-(s_2+\dots+s_r)} \\
\vdots \\
- \mathfrak L(s_r) \Omega^{-s_r} 
\end{pmatrix}
\end{align*}

In particular, for $1 \leq \ell \leq r$, the $(d_1+\dots+d_\ell)$th coordinate of the $\Log_{E'^*}^!(\bv^*_\fs)$ equals $- \Gamma_{s_\ell} \dots \Gamma_{s_r} \zeta_A(s_r\dots,s_\ell)$.
\end{theorem}
We note that we deduce easily from the above theorem logarithmic interpretations for $\nu$-adic MZV's (see \S \ref{sec: v-adic} for related definitions).

\begin{theorem}[Theorem \ref{theorem: nu adic}]
For $\fs=(s_1,\ldots,s_r) \in \N^r$, we put $d_\ell:=s_\ell+\dots+s_r$ for $1 \leq \ell \leq r$. Let $\nu$ be a finite place of $K$. Then there exists a nonzero $a \in A$ for which the series $\Log_{E'^*}(E'^*_a \bv^*_\fs)$ converges $\nu$-adically in $\Lie_{E'^*}(\C_\nu)$ and the $d_1$th coordinate of $\Log_{E'^*}(E'^*_a \bv^*_\fs)$ equals $-a \Gamma_{s_1} \ldots \Gamma_{s_r} \zeta_A(s_r,\dots,s_1)$.
\end{theorem}

In particular, we can define $\zeta_A(\fs)_\nu$ to be the value $-\frac{1}{a}$ multiplied by the $d_1$th coordinate of $\Log_{E'^*}(E'^*_a \bv^*_\fs)_\nu$. As a consequence we simplify some arguments of the proof of Chang and Mishiba \cite[\S 6.4]{CM19b} of a conjecture of Furusho over function fields stated as follows: if we denote by $\overline{\mathcal Z}_n$ (resp. $\overline{\mathcal Z}_{n,\nu}$) the $\overline K$-vector space generated by all $\infty$-adic (resp. $\nu$-adic) MZV's of weight $n$, the we have a well-defined surjective $\overline K$-linear map
	\[ \overline{\mathcal Z}_n \to \overline{\mathcal Z}_{n,\nu}, \quad \zeta_A(\fs) \mapsto \zeta_A(\fs)_\nu. \] 

In \S \ref{sec: examples} we provide examples to illustrate our results and compare our work with the works of Anderson-Thakur \cite{AT90} and Chang-Mishiba \cite{CM19b}. Compared to Chang-Mishiba's construction, ours is much more direct, has smaller dimension (see Proposition \ref{prop: dimension}) and is in the same spirit of \cite{AT90} as illustrated in \S \ref{sec: relation with CM}. In \S \ref{sec: relation with AT} we present further examples inspired by those given in \cite{AT90}.

\medskip

\noindent {\bf Acknowledgments.} 

We are grateful to Chieh-Yu Chang and Federico Pellarin for carefully reading the first version of this manuscript and for offering numerous insightful comments and suggestions which greatly improve the content and the exposition of the paper. We thank Bruno Anglès, Yoshinori Mishiba and Jing Yu for useful suggestions and remarks.

The second author (T. ND.) was partially supported by CNRS IEA "Arithmetic and Galois extensions of function fields"; the ANR Grant COLOSS ANR-19-CE40-0015-02 and the Labex MILYON ANR-10-LABX-0070.


\section{Anderson $t$-modules and dual $t$-motives}

In this section we briefly review the basic theory of Anderson $t$-modules and dual $t$-motives and the relation between them. We refer the reader to \cite[\S 5]{HJ20} for more details.

\subsection{Notation} ${}$\par

In this paper we will use the following notation.
\begin{itemize}
\item $\mathbb N=\{1,2,\dots\}$: the set of positive integers.
\item $\mathbb Z^{\geq 0}=\{0,1,\dots\}$: the set of non-negative integers.
\item $\mathbb Z$: the set of integers.
\item $\mathbb F_q$: a finite field having $q$ elements.
\item $p$: the characteristic of $\mathbb F_q$.
\item $\theta,t$: independent variables over $\mathbb F_q$.
\item $A$: the polynomial ring $\mathbb F_q[\theta]$.
\item $A_+$: the set of monic polynomials in $A$.
\item $A_{+,d}$: the set of monic polynomials in  $A$ of degree $d$ for $d\in \mathbb N$.
\item $K=\mathbb F_q(\theta)$: the fraction field of $A$.
\item $\infty$: the unique place of $K$ which is a pole of $\theta$.
\item $v_\infty$: the discrete valuation on $K$ corresponding to the place $\infty$ normalized such that $v_\infty(\theta)=-1$.
\item $\lvert\cdot\rvert_\infty= q^{-v_\infty}$: an absolute value on $K$.
\item $K_\infty=\mathbb F_q((\frac{1}{\theta}))$: the completion of $K$ at $\infty$.
\item $\mathbb C_\infty$: the completion of a fixed algebraic closure $\overline K_\infty$ of $K_\infty$. The unique valuation of $\mathbb C_\infty$ which extends $v_\infty$ will still be denoted by $v_\infty$.
\end{itemize}

\subsection{Review of Anderson $t$-modules} ${}$\par 

Let $R$ be an $\Fq$-algebra and let $R[\tau]$ denote the (non-commutative) skew-polynomial ring with coefficients in $R$, subject to the relation for $r\in R$,
\[\tau r = r^q\tau.\]
We similarly define $R[\sigma]$, but we require additionally that $R$ must be a perfect ring, now subject to the relation
\[\sigma r = r^{1/q} \sigma.\]

We define Frobenius twisting on $R[t]$ by setting for $g = \sum_j c_j t^j \in  R[t]$, 
\begin{equation*} 
g\twisti = \sum_j c_j^{q^i} t^j.
\end{equation*}
We extend twisting to matrices in $\Mat_{i\times j}(R[t])$ by twisting coordinatewise.  

\begin{definition}
Let $R$ be an $\Fq$-algebra equipped with an $\Fq$-algebra homomorphism $i:A \to R$. 

1) A $d$-dimensional Anderson $t$-module over $R$ is an $\F_q$-algebra homomorphism $E:A\to \Mat_d(R)[\tau]$, such that for each $a\in A$, 
\[E_a = d[a] + E_{a,1} \tau + \dots ,\quad E_{a,i} \in \Mat_d(R)\]
where $d[a] = i(a)I_d + N$ for some nilpotent matrix $N\in \Mat_d(R)$ (depending on $a$).  

2) A Drinfeld module is a one-dimensional Anderson $t$-module $\rho:A \to R[\tau]$. 
\end{definition}

For the rest of this paper, we will drop $i$ when no confusion results. Anderson $t$-modules will sometimes be called $t$-modules.

The map $d[\cdot]:A \lra \Mat_d(\overline K)$ is a ring homomorphism which extends naturally to $d[\cdot]:K \lra \Mat_d(\overline K)$ and describes the $\Lie$ action of $E$.  Note that there is an implicit dependence of the map $d[\cdot]$ on the $t$-module $E$ which we omit, since it does not cause any confusion. Let $E$ be an Anderson $t$-module of dimension $d$ over $R$ as above and let $B$ be an $R$-algebra. We can define two $A$-module structures on $B^d$. The first one is denoted by $E(B)$ where $A$ acts on $B^d$ via $E$:
	\[ a \cdot \begin{pmatrix} b_1\\ \vdots\\b_d \end{pmatrix}= d[a] \begin{pmatrix} b_1\\ \vdots\\b_d \end{pmatrix} + \sum_{k\geq 1} E_{a,k}  \begin{pmatrix} b_1^{q^k}\\ \vdots\\b_d^{q^k}\end{pmatrix}, \quad \text{for } a\in A, \, \begin{pmatrix} b_1\\ \vdots\\b_d\end{pmatrix} \in B^d. \]
The second one is denoted by $\Lie_E(B)$ where $A$ acts on $B^d$ via $d[\cdot]$:
	\[ a \cdot \begin{pmatrix} b_1\\ \vdots\\b_d \end{pmatrix}= d[a] \begin{pmatrix} b_1\\ \vdots\\b_d \end{pmatrix}, \quad \text{for } a\in A, \, \begin{pmatrix} b_1\\ \vdots\\b_d\end{pmatrix} \in B^d. \]

From now on, we will always work with Anderson $t$-modules over $R$ such that $R \subset \C_\infty$. Let $E:A\to \Mat_d(\C_\infty)[\tau]$ be an Anderson module of dimension $d$ over $\C_\infty$. We define $\Exp_E$ to be the exponential series associated to $E$, which is the unique function on $\C_\infty^d$ such that as an $\F_q$-linear power series we can write
\begin{equation*} 
\Exp _E(\bz) = \sum_{i=0}^\infty Q_i \bz\twisti,\quad Q_i\in \Mat_d(\C_\infty), \bz\in \C_\infty^d,
\end{equation*}
with $Q_0 = I_d$ and such that for all $a\in A$ and $\bz\in \C_\infty^d$,
\begin{equation*} 
\Exp_E(d[a]\bz) = E_a(\Exp_E(\bz)).
\end{equation*}

The logarithm function $\Log_E$ is then defined as the formal power series inverse of $\Exp_E$.  We denote its power series as
\begin{equation*} 
\Log _E(\bz) = \sum_{i=0}^\infty P_i \bz\twisti,\quad P_i\in \Mat_d(\C_\infty), \bz\in \C_\infty^d.
\end{equation*}
We note that as functions on $\C_\infty^d$ the function $\Exp_E$ is everywhere convergent, whereas $\Log_E$ has some finite radius of convergence.

\subsection{Units and Stark units} ${}$\par \label{sec: units}

We define the Tate algebra $\TT$ over $\C_\infty$ as the space of power series in $t$ which converge on the disc of radius $1$, in other words,
\begin{equation*} 
   \TT := \left \{ \sum_{i=0}^\infty b_i t^i \in \power{\C_\infty}{t} \biggm| \big\lvert b_i  \big\rvert_{\infty} \to 0 \right \}.
\end{equation*}
We denote by $\bL$ the fraction field of $\TT$.

Define the Gauss norm $\lVert \cdot \rVert$ on $\TT$ by setting \[ \lVert f \rVert:=\max_{i} \left\{ |b_{i}|_{\infty}  \right\} \]
  for $f=\sum_{i \geq 0} b_{i}t^{i}\in \TT$. We then extend the Gauss norm to $\Mat_{\ell \times m}\left (\TT\right )$ by setting
	\[\lVert B \rVert = \max_{i,j} \left\{ \lVert B_{ij} \rVert \right\}\] 
for $B=(B_{ij}) \in  \Mat_{\ell \times m}\left (\TT\right )$.

In what follows we fix an Anderson $t$-module $E:A\to \Mat_d(\overline K)[\tau]$ of dimension $d$ over $\overline K$. Let $z$ be an indeterminate with $\tau z=z \tau$ and let $\bT_z(\C_\infty)$ be the Tate algebra in the variable $z$ with coefficients in $\C_\infty$. We define the canonical $z$-deformation of the $t$-module $E$ denoted by $\widetilde{E}$ to be the homomorphism of $\Fq[z]$-algebras $\widetilde{E}: A[z] \to \Mat_d(\overline K[z])[\tau]$ such that 
	\[ \widetilde{E}_a =\sum_{k\geq 0} E_{a,k} z^k \tau ^k, \quad a \in A. \]
Then there exists a unique series $\Exp_{\widetilde{E}}\in I_d+\tau \Mat_d(\overline K[z])[[\tau ]]$ such that
	\[ \Exp_{\widetilde{E} } d[a] = \widetilde{E}_a \Exp_{\widetilde{E} }, \quad a \in A, \]
(see \cite[\S 3]{APTR16} for more details).  One can show that if $\Exp_E=\sum_{i\geq 0} Q_i \tau ^i$, then $\Exp_{\widetilde{E} }=\sum_{i\geq 0} Q_i z^i \tau ^i$. In particular, $\Exp_{\widetilde{E} }$ converges on $\Lie_{\widetilde{E}}(\bT_z(\C_\infty))$ and induces  a homomorphism of $A[z]$-modules 
	\[ \Exp_{\widetilde{E} }:\Lie_{\widetilde{E}}(\bT_z(\C_\infty))\to \widetilde{E}(\bT_z(\C_\infty)). \]
We denote by $\Log_{\widetilde{E}} \in I_d+\tau \Mat_d(\overline K[z])[[\tau ]]$ the inverse of $\Exp_{\widetilde{E}}$. Similarly, if $\Log_E=\sum_{i\geq 0} P_i \tau ^i$, then $\Log_{\widetilde{E} }=\sum_{i\geq 0} P_i z^i \tau ^i$.

We denote by $\ev: \Lie_{\widetilde{E}}(\bT_z(\C_\infty))\to \Lie_E(\C_\infty)$ the evaluation map at $z=1$. If $\bx \in \Lie_{\widetilde{E}}(\bT_z(\C_\infty))$, then we also write $\bx \big|_{z=1}$ for $\ev(\bx)$.  Following \cite{ANDTR17a,APTR18,ATR17,Tae10} we introduce various notions of units and  of logarithmic identities for Anderson $t$-modules.

\begin{definition} \label{definition: log_St}
Letting $\bz \in \C_\infty^d$, we put $\mathbf{Z}:=\Exp_{E'}(\bz)$. 

1) Following Taelman \cite{Tae10} we say that $\bz$ is {\it a unit for $E'$} if $\mathbf{Z}=\Exp_{E'}(\bz) \in \overline K^d$. The latter equality is also called {\it a log-algebraicity identity for $E'$}.

2) We say that $\bz$ is {\it a Stark unit for $E'$} if we can write $\bz=\bx \big|_{z=1}$ for some $\bx \in \Lie_{\widetilde E'}(\bT_z(\C_\infty))$ satisfying $\Exp_{\widetilde E'}(\bx) \in \overline K[z]^d$.  We also say that we have {\it a Stark logarithmic identity for $\bz$}, and we write
	\[ \Log_{E'}^{St}(\mathbf{Z}) = \bz.\]
\end{definition}

\begin{remark} \label{rem:units vs Stark units}
1) We refer the reader to \cite{Tae10,Tae12a} for more details about arithmetic of units.

2) The first example of Stark units appeared in the pioneering work of Anderson \cite{And96} in which he introduced the analogue of cyclotomic units for the Carlitz module. Recently, based on the fundamental work of Pellarin in \cite{Pel12},  Anglès, Tavares Ribeiro and the second author have introduced and developed the theory of Stark units for Anderson modules (see \cite{ANDTR17a,ANDTR20a,ATR17}). This notion turns out to be a powerful tool for  investigating log-algebraic identities \cite{ANDTR17a,ANDTR20a,APTR18} as well as the class formula à la Taelman in full generality \cite{ANDTR20b}. 
\end{remark}

We note that if $\bz$ is a Stark unit for $E$, then it is also a unit for $E$. In fact, we set $\mathbf{Z}:=\Exp_E(\bz)$. By Definition \ref{definition: log_St} there exists $\bx \in \Lie_{\widetilde{E}}(\bT_z(\C_\infty))$ such that $\Exp_{\widetilde E}(\bx) \in \overline K[z]^d$ and $\bz=\bx \big|_{z=1}$. It follows that 
	\[ \mathbf{Z}=\Exp_E(\bz)=\Exp_{\widetilde E}(\bx) \big|_{z=1} \in \overline K^d. \]
Hence, $\bz$ is also a unit for $E$.
 	
\begin{remark}
1) We continue with the above notation. If we write the polynomial $\Exp_{\widetilde E}(\bx)=\sum_{i=0}^m \mathbf{Z}_i z^i$ with $\mathbf{Z}_i \in \overline K^d$ ($0 \leq i \leq m$), then the fact that $\bx \in \Lie_{\widetilde{E}}(\bT_z(\C_\infty))$ is equivalent to the following condition
	\[ P_k \mathbf{Z}_0^{(k)}+ \dots+P_{k-m} \mathbf{Z}_m^{(k-m)} \to 0 \quad \text{when } k \to +\infty. \]
Here we understand that $P_{k-i}=0$ if $k-i<0$. And we get
	\[ \bz=\sum_{k\geq 0} \left(P_k \mathbf{Z}_0^{(k)}+ \dots+P_{k-m} \mathbf{Z}_m^{(k-m)} \right). \]
In other words, $\bz$ is a kind of re-indexed logarithms as already observed in \cite{Gre17b,GP18,Tha92}.

2) If the polynomial $\Exp_{\widetilde E}(\bx)$ is a monomial, then we express $\Exp_{\widetilde E}(\bx)= \mathbf{Z}_i z^i$ for some $i \geq 0$. It is clear that $\mathbf{Z}_i$ lies in the domain of convergence of $\Log_E$ and $\bz$ is a logarithm: 
	\[ \bz=\Log_E(\mathbf{Z}_i). \]
\end{remark}

\subsection{Review of dual $t$-motives} ${}$\par \label{sec: dual motives}

We briefly review the notion of dual $t$-motives and explain the relation with $t$-modules thanks to Anderson  (see \cite[\S 4]{BP20} and \cite[\S 5]{HJ20} for more details).

\begin{definition} 
An effective dual $t$-motive is a $\overline K[t,\sigma]$-module $\calM'$ which is free and finitely generated over $\overline K[t]$ such that for $\ell\gg 0$ we have
	\[(t-\theta)^\ell(\calM'/\sigma \calM') = \{0\}.\]
\end{definition}

\begin{remark}
1) We mention that effective dual $t$-motives are called Frobenius modules in \cite[\S 2.2]{CPY19}. 

\noindent 2) Note that Hartl and Juschka \cite[\S 4]{HJ20} introduced a more general notion of dual $t$-motives. In particular, effective dual $t$-motives are always dual $t$-motives.
\end{remark}

Throughout this paper we will always work with effective dual $t$-motives. Therefore, we will sometimes drop the word "effective" where there is no confusion.

Let $\calM'$ and $\calM''$ be two effective dual $t$-motives. Then a morphism of effective dual $t$-motives $\calM' \to \calM''$ is just a homomorphism of left $\overline K[t,\sigma]$-modules. We denote by $\cF$ the category of effective dual $t$-motives equipped with the trivial object $\mathbf{1}$. 

We say that an object $\calM'$ of $\cF$ is given by a matrix $\Phi' \in \Mat_r(\overline K[t])$ if $\calM'$ is a $\overline K[t]$-module free of rank $r$ and the action of $\sigma$ is represented by the matrix $\Phi'$ on a given  $\overline K[t]$-basis for $\calM'$.

Recall that $\bL$ denotes the fraction field of the Tate algebra $\TT$. We say that an object $\calM'$ of $\cF$ is uniformizable or rigid analytically trivial if there exists a matrix $\Psi' \in \text{GL}_r(\bL)$ satisfying $\Psi'^{(-1)}=\Phi' \Psi'$. The matrix $\Psi'$ is called a rigid analytic trivialization of $\calM'$. By \cite[Proposition 3.3.9]{Pap08} there exists a rigid analytic trivialization $\Psi_0'$ of $\calM'$ with $\Psi_0' \in \text{GL}_r(\TT)$. Further, if $\Psi'$ is a rigid analytic trivialization of $\calM'$, then $\Psi'=\Psi_0' B$ with $B \in \Mat_r(\Fq(t))$.

In what follows, let $\calM'$ be an effective dual $t$-motive of rank $r$ over $\overline K[t]$ which is also free and finitely generated of rank $d$ over $\overline K[\sigma]$. Let $\bm=\{m_1,\dots,m_r\}$ denote a $\overline K[t]$-basis for $\calM'$ and let $\bw=\{w_1,\dots,w_d\}$ denote a $\overline K[\sigma]$-basis for $\calM'$. Using the basis $\bm=\{m_1,\dots,m_r\}$, we identify 
$\overline K[t]^r$ with $\calM'$ by the map
\begin{equation}\label{eq:iota m}
\iota_\bm:\overline K[t]^r \to \calM', \quad (g_1,\ldots,g_r)^\top \mapsto g_1 m_1+\ldots+g_r m_r.
\end{equation}
We extend $\iota_\bm$ to Tate algebras still denoted by $\iota_\bm:\bT^r \to \calM' \otimes_{\overline K[t]} \bT$. 

Similarly, using the basis $\bw=\{w_1,\dots,w_d\}$, we also identify 
$\overline K[\sigma]^d$ with $\calM'$
\begin{equation}\label{eq:iota sigma}
\iota_\bw:\overline K[\sigma]^d \to \calM', \quad (h_1,\ldots,h_d)^\top \mapsto h_1 w_1+\ldots+h_d w_d.
\end{equation}
Letting $\iota=\iota\inv_\bw \circ \iota_\bm$, we get the map 
\begin{equation}\label{eq:iotamap}
\iota:\overline K[t]^r \to \overline K[\sigma]^d
\end{equation}
which ``switches" between these bases.

Once we fix the $\overline K[t]$-basis $\bm$, then there exists some matrix $\Phi' \in \Mat_r(\overline K[t])$ such that $\sigma$ acts on $\overline K[t]^r$ by inverse twisting and right multiplication by $\Phi'$ --- or we may transpose to get a left multiplication:
\[\sigma\begin{pmatrix}
g_1\\
\vdots \\
g_r
\end{pmatrix} = \Phi'^\top \begin{pmatrix}
g_1\\
\vdots \\
g_r
\end{pmatrix}\twistinv,\quad g_i \in \overline K[t].\]
We note that this $\sigma$-action extends to $\TT^r \isom M \otimes _{{\overline K}[t]} \TT$ in the natural way.  
	
We recall the definition of the maps 
	\[\delta_0:\calM' \to \overline K^d, \quad \delta_1:\calM' \to \overline K^d \]
from \cite[Proposition 5.6]{HJ20}. Letting $m \in \calM'$, we write
	\[ m = c_{0,1}w_1 + \dots + c_{0,d}w_d + c_{1,1}\sigma(w_1) + \dots + c_{1,d}\sigma(w_d) + \dots, \quad c_{i,j} \in \overline K.\]
We set
\begin{equation} \label{D:deltaextendedmaps}
\delta_0(m) = \begin{pmatrix}
c_{0,1}\\
\vdots \\
c_{0,d}
\end{pmatrix},\quad
\delta_1(m) = \begin{pmatrix}
c_{0,1}\\
\vdots \\
c_{0,d}
\end{pmatrix} + 
\begin{pmatrix}
c_{1,1}\\
\vdots \\
c_{1,d}
\end{pmatrix}\twist
+
\cdots.
\end{equation}
Similarly, letting $z$ be a variable, we define the $z$-version $\delta_{1,z}$ of the map $\delta_1$ by
\begin{equation*}
\delta_{1,z}(m) = \begin{pmatrix}
c_{0,1}\\
\vdots \\
c_{0,d}
\end{pmatrix} + 
\begin{pmatrix}
c_{1,1}\\
\vdots \\
c_{1,d}
\end{pmatrix}\twist z
+
\begin{pmatrix}
c_{2,1}\\
\vdots \\
c_{2,d}
\end{pmatrix}^{(2)} z^2
+
\cdots
.
\end{equation*}

We then observe that the kernel of $\delta_1$ equals $\sigma-1$, and thus can write the commutative diagram
\begin{align*} 
\begin{CD}
\calM'/(\sigma-1)\calM'  @>{\delta_1}>> \overline K^d \\
@V{a(t)}VV @VV{E'_a}V \\
\calM'/(\sigma-1)\calM' @>{\delta_1}>> \overline K^d
\end{CD}
\end{align*}
where the left vertical arrow is multiplication by $a(t)$ and the right vertical arrow is the map induced by multiplication by $a$, which we denote by $E'_a$.  By \cite[Proposition 5.6]{HJ20}, $E'$ defines an Anderson $t$-module over $\overline K$, and we call this the Anderson $t$-module associated with $\calM'$.  Thus we have canonical isomorphisms of $\Fq[t]$-modules
	\[ \calM'/\sigma \calM' \iso \Lie_{E'}(\overline K), \]
and
	\[ \calM'/(\sigma-1)\calM' \iso E'(\overline K). \] 
	
Note that the map $\delta_0:\calM' \lra \overline K^d$ factors through $\calM'/(t-\theta)^d \calM'$. Thus it extends to $\calM'\otimes _{{\overline K}[t]} \TT$ and  $\calM'\otimes _{{\overline K}[t]} \overline K[t]_{(t-\theta)}$ in the natural way where $\calM'\otimes _{{\overline K}[t]} \overline K[t]_{(t-\theta)}$ denotes the localization of $\overline K[t]$ outside the prime ideal $t-\theta$ of $\overline K[t]$ (see \cite[Proposition 5.6]{HJ20}).

\begin{remark}
Anderson showed that the functor $\calM' \mapsto E'$ gives rise to an equivalence from the category of effective dual $t$-motives $\calM'$ that are free and finitely generated as $\overline K[\sigma]$-modules onto the full subcategory of so-called $A$-finite Anderson $t$-modules (see for example \cite[Theorem 5.9]{HJ20}).
\end{remark}


\section{The main result} \label{sec: main theorem}

This section aims to prove the main result of this paper (see Theorem \ref{theorem: main theorem}). We establish refinements of log-algebraic identities for Anderson $t$-modules which provide a general framework for many earlier results which have been proven in a somewhat ad-hoc fashion. Finally we discuss relations with Anderson's analytic theory of $A$-finite $t$-modules and emphasize the advantage of Stark units compared to units. 

\subsection{$\textrm{Ext}^1$-modules and $t$-modules} ${}$\par

In this section we explain a deep connection due to Anderson between some $\textrm{Ext}^1$-modules and Anderson $t$-modules. We follow closely the presentation given in \cite[\S 5.2]{CPY19}. 

In what follows, we fix $\calM'$ to be an effective dual $t$-motive of rank $r$ over $\overline K[t]$. Recall the definitions of $\bm=\{m_1,\dots,m_r\}$, $\bw=\{w_1,\dots,w_d\}$, $\iota_\bm$, $\iota_\bw$, $\iota$, $\delta_0$ and $\delta_1$  from \S \ref{sec: dual motives}.  Composing with the map $\iota_\bm$ defined in \eqref{eq:iota m}, we get three maps $\delta_0 \circ \iota_\bm:\overline K[t]^r \to \overline K^d$, $\delta_1 \circ \iota_\bm:\overline K[t]^r \to \overline K^d$ and $\delta_{1,z} \circ \iota_\bm:\overline K[t]^r \to \overline K[z]^d$.
From now on, to avoid heavy notation, we still denote these maps by 
	\[ \delta_0:\overline K[t]^r \to \overline K^d \] 
and 
	\[ \delta_1:\overline K[t]^r \to \overline K^d, \quad \delta_{1,z}:\overline K[t]^r \to \overline K[z]^d. \] 

We denote by $\Phi' \in \Mat_r(\overline K[t])$ its associated matrix. Based on an unpublished work of Anderson \cite[Theorem 5.2.1]{CPY19} (see also \cite{Tae20}), one shows that if $\bm=(m_1,\dots,m_r)$ is a $\overline K[t]$-basis of $\calM'$ on which the $\sigma$-action is represented by the matrix $\Phi'$ and $\calM \in \textrm{Ext}^1_\cF(\mathbf{1},\calM')$ is the dual $t$-motive given by the matrix
\begin{align*}
\Phi=\begin{pmatrix}
\Phi' & 0  \\
\bff & 1 
\end{pmatrix}, \quad \text{with } \bff=(f_1,\dots,f_r)\in \Mat_{1\times r}(\overline K[t]),
\end{align*}
then the map 
\begin{align} \label{eq: Ext}
\alpha: \textrm{Ext}^1_\cF(\mathbf{1},\calM') & \to \calM'/(\sigma-1)\calM' \\
\calM & \mapsto f_1 m_1 + \dots +f_r m_r \notag
\end{align}
is an isomorphism of $\Fq[t]$-modules. 

For such an extension $\calM \in \textrm{Ext}^1_\cF(\mathbf{1},\calM')$, we know that $\calM$ is uniformizable by \cite[Lemma 4.20]{HJ20}. By \cite[Proposition 3.3.9]{Pap08} there exists a rigid analytic trivialization $\Psi \in \Mat_{r+1}(\bL)$ of $\Phi$ such that if we set $\Upsilon:=\Psi\inv$, then we require that $\Upsilon \in \Mat_{r+1}(\bT)$. We put 
\begin{align*}
\Psi=\begin{pmatrix}
\Psi' & 0  \\
\Psi_{r+1} & 1 
\end{pmatrix} \in \text{GL}_{r+1}(\bL),
\end{align*}
and
\begin{align} \label{eq: inverse of Psi}
\Upsilon=\begin{pmatrix}
\Upsilon' & 0  \\
\Upsilon_{r+1} & 1 
\end{pmatrix} \in \Mat_{r+1}(\TT)
\end{align}
where 
	\[ \Psi_{r+1}=(\Psi_{r+1,1},\ldots,\Psi_{r+1,r}) \in \Mat_{1 \times r}(\bL) \] 
and 
	\[ \Upsilon_{r+1}=(\Upsilon_{r+1,1},\ldots,\Upsilon_{r+1,r}) \in \Mat_{1 \times r}(\bT). \] 
In particular, $\Psi'$ is a rigid analytic trivialization matrix, i.e. $\Psi\twistinv = \Phi \Psi$.

\begin{remark} \label{remark:denominator}
By \cite[Proposition 3.3.9]{Pap08} again, we know that there exists a polynomial $F \in \Fq[t]$ such that $F \Psi \in \Mat_{r+1}(\bT)$.
\end{remark}

Inspired by \cite{CPY19} we define the point $\bv_{\calM} \in E'(\overline K)$ by the image of $\calM$ via the composition of isomorphisms
\begin{equation}\label{D:vMdef}
\delta_1 \circ \alpha: \textrm{Ext}^1_\cF(\mathbf{1},\calM') \iso \calM'/(\sigma-1)\calM' \iso E'(\overline K). 
\end{equation} 
Thus $\bv_{\calM}=\delta_1   (\bff^\top) \in E'(\overline K)$. We also set 
	\[ \bv_{\calM,z}:=\delta_{1,z}   (\bff^\top) \in \widetilde{E}'(\overline K) \] 
where $\widetilde{E}'$ is the $z$-deformation $t$-module attached to $E'$ (see \S \ref{sec: units}).

To end this section we mention that $\textrm{Ext}^1_\cF(\mathbf{1},\calM')$ has a natural $\Fq[t]$-module structure defined as follows. Let $\calM_1$ and $\calM_2$ be two objects of $\textrm{Ext}^1_\cF(\mathbf{1},\calM')$ defined by the matrices 
\begin{align*}
\Phi_1=\begin{pmatrix}
\Phi' & 0  \\
\bv_1 & 1 
\end{pmatrix} \in \Mat_{r+1}(\overline K[t]), \quad \bv_1 \in \Mat_{1 \times r}(\overline K[t]),
\end{align*}
and 
\begin{align*}
\Phi_2=\begin{pmatrix}
\Phi' & 0  \\
\bv_2 & 1 
\end{pmatrix} \in \Mat_{r+1}(\overline K[t]), \quad \bv_2 \in \Mat_{1 \times r}(\overline K[t]).
\end{align*}
Then for any $a_1,a_2 \in \Fq[t]$, $a_1*\calM_1+a_2*\calM_2$ is defined to be the class in  $\textrm{Ext}^1_\cF(\mathbf{1},\calM')$ represented by
\begin{align*}
\begin{pmatrix}
\Phi' & 0  \\
a_1\bv_1+a_2\bv_2 & 1 
\end{pmatrix} \in \Mat_{r+1}(\overline K[t]).
\end{align*}

\begin{remark} \label{rem:torsion class}
Let $\calM$ be a class in $\textrm{Ext}^1_\cF(\mathbf{1},\calM')$. Let $E'$ be the $t$-module attached to $\calM'$ and $\bv_{\calM} \in E'(\overline K)$ be the special point attached to $\calM$ as above. We observe that $\calM$ is a torsion class in $\textrm{Ext}^1_\cF(\mathbf{1},\calM')$ if and only if $\bv_{\calM}$ is a torsion point in $E'(\overline K)$.
\end{remark}

\subsection{Statement of the Main Result} ${}$\par

We keep the above notation. We give a definition which simplifies notation enormously throughout the paper.

\begin{definition} \label{definition: split log}
Given a $d$-dimensional $t$-module $E$ over $\C_\infty$ with logarithm function $\Log_E$ and two points $\bz, \mathbf{Z} \in \C_\infty^d$, we say that we have a split-logarithmic identity (for $\bz$)
\[\Log_E^!(\mathbf{Z}) = \bz\]
if there exists some finite collection of pairs $\{(a_i,\bu_i)\}\subset A\times \C_\infty^d$ where each $\bu_i$ is in the domain of convergence of $\Log_E$, such that
\begin{align*}
\bz &= \sum_{i} d[a_i] \Log_E(\bu_i), \\
\mathbf{Z} &= \sum_{i} E_{a_i} \bu_i.
\end{align*}
\end{definition}

\begin{remark} \label{remark: split log identity}
1) We note that if $\Log_E^!(\mathbf{Z}) = \bz$, then $\Log_E^{\text{St}}(\mathbf{Z}) = \bz$ and $\Exp_E(\bz)=\mathbf{Z}$.   Further, each $\Log_E(\bu_i)$ is a Stark unit for $E$. This implies that $\bz$, which is a linear combination of Stark units with coefficients in $A$ (via the action $a \mapsto d[a]$), is also a Stark unit for $E$. 

2) Split-logarithmic identities are common when one discovers a log-algebraic identity of the form $\Exp_E(\bz) = \mathbf{Z}$, but $\mathbf{Z}$ is not inside the domain of convergence of $\Log_E$.  In some cases one can decompose $\mathbf{Z}$ into a sum of terms $E_{a_i}\bu_i$ as above, such that each $\bu_i$ is inside the domain of convergence of $\Log_E$.  Such is the case in the celebrated log-algebraicity theorem of Anderson and Thakur for tensor powers of the Carlitz module \cite[Theorem 3.8.3]{AT90} (see also \cite[\S II]{Tha92}).  

3) We comment that each time we give a split-logarithmic identity in this paper, the exact linear combination of $(a_i,\bu_i)$ is given explicitly in the proof.  Thus there is nothing mysterious about these split-logarithmic identities, they are merely a tool to unify notation.
\end{remark}

We are ready to state the main result of this paper which provides log-algebraic identities for Anderson $t$-modules.
	
\begin{theorem} \label{theorem: main theorem}
We keep the above notation and let $\Log^{\text{St}}_{E'}$ and $\Log^!_{E'}$ be defined in as Definitions \ref{definition: log_St} and \ref{definition: split log}, respectively. 

Let $F \in \Fq[t]$ such that $F \Psi \in \Mat_{r+1}(\bT)$ (see Remark \ref{remark:denominator}). We suppose that $F \Psi_{r+1}^{(k)}$ converges to $0$ as a vector of functions in $\Mat_{1\times r}(\TT)$.  
\begin{enumerate}
\item[(a)] Let $\Upsilon_{r+1}$ be defined as in \eqref{eq: inverse of Psi}. Then 
\begin{enumerate}
\item[(a1)] The point $\delta_0   (\bff^\top-\Upsilon_{r+1}^\top)$ is a Stark unit for $E'$.

\item[(a2)] We have a Stark logarithmic identity
	\[ \Log^{\text{St}}_{E'}(\bv_{\calM})=\delta_0   (\bff^\top-\Upsilon_{r+1}^\top). \]
\end{enumerate}

\item[(b)] Let $\alpha$ be the map defined in \eqref{eq: Ext}. Suppose that there exists some finite collection of triples $\{(\ell_i,n_i,\bu_i=(u_{i,1},\ldots,u_{i,d})^\top)\}\subset \mathbb Z^{\geq 0} \times \mathbb Z^{\geq 0} \times \C_\infty^d$ where each $\bu_i$ is in the domain of convergence of $\Log_{E'}$, such that
	\[\alpha(\calM) = \sum_i t^{n_i} \sigma^{\ell_i} \left( \sum_{j=1}^d u_{i,j} w_j \right),\]
where $w_j$ elements of the $\overline K[\sigma]$-basis $\bw$.  Then we have a split-logarithmic identity
	\[\Log_{E'}^!(\bv_{\calM}) = \delta_0   (\bff^\top-\Upsilon_{r+1}^\top).  \]
\end{enumerate}
If we additionally have that $\delta_0   (\bff^\top)=0$, then the right-hand side of the main equations in Parts (a) and (b) above is simply given by $\delta_0   (-\Upsilon_{r+1}^\top)$.
\end{theorem}

\begin{remark}
1) By Remark \ref{remark: split log identity}, Part (b) could be considered as a refinement of Part (a).

2) It is clear that the condition that $F \Psi_{r+1}^{(k)}$ converges to $0$ as a vector of functions in $\Mat_{1\times r}(\TT)$ does not depend on the choice of $F \in \Fq[t]$. In particular, when $\Psi \in \text{GL}_{r+1}(\TT)$, we could take $F=1$ as we will see in the next sections.
\end{remark}

\begin{remark}
We mention below some known examples of Theorem \ref{theorem: main theorem}.

1) As mentioned before, Anderson and Thakur \cite{AT90} gave split-logarithmic identities for Carlitz zeta values.

2) Chang, Mishiba and the first author gave split-logarithmic identities for Carlitz multiple star polylogarithms (see \cite{CGM19,CM19b} and \S \ref{sec: CM model} for more details).

3) For higher genus curves, Thakur studied special zeta values associated to rings $A$ such that $A$ is principal. For such rings, he obtained both Stark logarithmic identity and split-logarithmic identity for special zeta values at $1$ (see \cite[\S II]{Tha92}).

4) For elliptic curves, Stark logarithmic identities for special zeta values can be obtained using minor adjustments to \cite{Gre17b,GND20,GP18}. However, it seems very difficult to obtain split-logarithmic identities for these values (see \cite[Remark 6.4]{Gre17b}).
\end{remark}

\subsection{Proof of the main theorem: Part (a)} ${}$\par

In this section we prove Theorem \ref{theorem: main theorem}, Part (a). The proof is divided into several steps.

\noindent {\bf Step 1.} We compute the coefficients of $\Log_{E'}$. We set
\begin{align*}
\Theta:=(\Phi\inv)^\top \in \Mat_{r+1}(\overline K(t)),
\end{align*}
and 
\begin{align*}
\Theta':=({\Phi'}\inv)^\top \in \Mat_r(\overline K(t)).
\end{align*}
Now if we write 
	\[ \Log_{E'}=\sum_{n \geq 0} P_n \tau^n, \]
then
	\[ \Log_{\widetilde{E}'}=\sum_{n \geq 0} P_n z^n \tau^n. \]
	
By \cite[Proposition 2.2]{ANDTR20a}, for $n \geq 0$, the $n$th coefficient of the logarithm series of $E'$ is given as follows. Let $\bv=(v_1,\dots,v_d)^\top \in \overline K^d$. Letting $m:=\iota_\bw(\bv)=v_1 w_1 + \dots + v_d w_d$, we see that $m$ belongs to $\calM'$. Thus we can express it in the $\overline K[t]$-basis $\{m_1,\dots,m_r\}$ using the map $\iota$ from \eqref{eq:iotamap}
	\[ \iota\inv(\bv)=\iota\inv(v_1,\dots,v_d)^\top = (g_1,\dots,g_r)^\top\in \overline K[t]^r. \]
In other words, $m=v_1 w_1 + \dots + v_d w_d=g_1 m_1+\ldots+g_r m_r$. Then by \cite[Proposition 2.2]{ANDTR20a} (see also \cite[Lemma 4.2.1]{CGM19} for an explicit example of this) we have
	\[ P_n \bv^{(n)} = \delta_0   (\Theta'^{(1)} \dots \Theta'^{(n)} \iota\inv (\bv)^{(n)}). \]

\noindent {\bf Step 2.} We recall that 
	\[ \bv_{\calM,z}:=\delta_{1,z}    (\bff^\top) \in \widetilde{E}'(\overline K[z]), \]
and 
	\[ \bv_{\calM}={\bv_{\calM,z}} \big|_{z=1}. \] 
This means that if we write
	\[ f_1 m_1 + \dots f_r m_r = v_{0,1}w_1 + \dots + v_{0,d}w_d + v_{1,1}\sigma(w_1) + \dots + v_{1,d}\sigma(w_d) + \dots, \quad \text{with } v_{i,j} \in \overline K, \]
and set $\bv_i = (v_{i,1},\dots,v_{i,d})^\top$, then we get
\begin{equation*}
\bv_{M,z}=\delta_{1,z}(\bff^\top) =\bv_0 + \bv_1\twist z + \bv_2 \twistk{2} z^2 + \dots.
\end{equation*}

Let $\bv=(v_1,\ldots,v_d)^\top \in \Mat_{d \times 1}(\overline K)$. By \eqref{eq:iota sigma}, $\bv$ can be identified as an element $\iota_\bw(\bv)=v_1w_1+\ldots+v_dw_d$ of $\calM'$. We recall that $\iota\inv_\bm(\sigma(\iota_\bw(\bv)))=\Phi'^\top \iota\inv(\bv)^{(-1)}$. Then we get an equality of formal series in $z$ (we will interpret this identity in a Tate algebra under certain conditions in the Step 3 of the proof)
\begin{align*}
\Log_{\widetilde{E}'}(\bv z) &=\sum_{n \geq 0} \delta_0   (\Theta'^{(1)} \dots \Theta'^{(n)} \iota\inv(\bv)^{(n)}) z^{n+1}  \\
&=\sum_{n \geq 0} \delta_0   (\Theta'^{(1)} \dots \Theta'^{(n)} \Theta'^{(n+1)} \Phi'^{\top (n+1)} \iota\inv(\bv)^{(n)}) z^{n+1}  \\
&=\sum_{n \geq 0} \delta_0   (\Theta'^{(1)} \dots \Theta'^{(n)} \Theta'^{(n+1)} \iota\inv_\bm(\sigma(\iota_\bw(\bv)))^{(n+1)}) z^{n+1}  \\
&=\sum_{n \geq 0} \delta_0   (\Theta'^{(1)} \dots \Theta'^{(n)} \iota\inv_\bm(\sigma(\iota_\bw(\bv)))^{(n)}) z^{n}.
\end{align*}
Here the second equality comes from the fact that $\Theta':=((\Phi')^{-1})^\top$, and the last one holds since $\delta_0(\sigma(\iota_\bw(\bv)))=0$.

More generally, by similar arguments we show that for $j \in \mathbb N$, 
\begin{equation} \label{eq: shift}
\Log_{\widetilde{E}'}(\bv z^j) =\sum_{n \geq 0} \delta_0   (\Theta'^{(1)} \dots \Theta'^{(n)} \iota\inv_\bm(\sigma^j(\iota_\bw(\bv)))^{(n)}) z^{n}.
\end{equation}

We claim that
\begin{align*}
\Log_{\widetilde{E}'}(\bv_{\calM,z}) &= \sum_{n \geq 0} \delta_0   (\Theta'^{(1)} \dots \Theta'^{(n)} (\bff^\top)^{(n)}) z^n.
\end{align*}
In fact, by \eqref{eq: shift} we obtain
\begin{align*}
\Log_{\widetilde{E}'}(\bv_{M,z}) &= \Log_{\widetilde{E}'} \left(\bv_0 + \bv_1\twist z + \bv_2 \twistk{2} z^2 + \dots
\right). \\
&=\sum_{n \geq 0} \sum_{j \geq 0} \delta_0 \left(\Theta'^{(1)} \dots \Theta'^{(n)} \iota\inv_\bm \left(\sigma^j \left(\iota_\bw(\bv_j^{(j)}) \right) \right)^{(n)} \right) z^{n} \\
&=\sum_{n \geq 0} \delta_0 \left(\Theta'^{(1)} \dots \Theta'^{(n)} \iota\inv_\bm \left(\sum_{j \geq 0} \sigma^j \left(\iota_\bw(\bv_j^{(j)}) \right) \right)^{(n)} \right) z^{n}.
\end{align*}
We analyze now the sum $\sum_{j \geq 0} \sigma^j \left(\iota_\bw(\bv_j^{(j)}) \right)$ viewed as an element of $\calM'$. We have
\begin{align*}
\sum_{j \geq 0} \sigma^j \left(\iota_\bw(\bv_j^{(j)}) \right) &=\sum_{j \geq 0} \sigma^j \left(v_{j,1}^{(j)} w_1+\ldots+v_{j,d}^{(j)} w_d \right) \\
&= \sum_{j \geq 0} (v_{j,1} \sigma^j (w_1)+\ldots+v_{j,d} \sigma^j(w_d) ) \\
&= f_1 m_1 + \dots f_r m_r.
\end{align*}
This implies
	\[ \iota\inv_\bm \left(\sum_{j \geq 0} \sigma^j \left(\iota_\bw(\bv_j^{(j)}) \right)\right)=\iota\inv_\bm(f_1 m_1 + \dots f_r m_r)=\bff^\top \]
and the claim follows immediately.	

\medskip

\noindent {\bf Step 3.} We recall that $\bff=(f_1,\dots,f_r)$ and $\Psi_{r+1}=(\Psi_{r+1,1} \dots,\Psi_{r+1,r})$. Since $\Upsilon'={\Psi'}\inv$, we get 
\begin{equation} \label{eq:Upsilon last row}
\Upsilon^\top_{r+1}=-\Upsilon'^{\top} \Psi_{r+1}^{\top}.
\end{equation}

The equality $\Psi^{(-1)}=\Phi \Psi$ implies
\begin{align*}
\begin{pmatrix}
\Psi'^{(-1)} & 0  \\
\Psi_{r+1}^{(-1)} & 1 
\end{pmatrix}
=\begin{pmatrix}
\Phi' & 0  \\
\bff & 1 
\end{pmatrix}
\begin{pmatrix}
\Psi' & 0  \\
\Psi_{r+1} & 1 
\end{pmatrix}
=\begin{pmatrix}
\Phi' \Psi' & 0  \\
\bff \Psi'+\Psi_{r+1} & 1 
\end{pmatrix}.
\end{align*}
Thus
	\[ \Psi_{r+1}^{(-1)} =\bff \Psi'+\Psi_{r+1}. \]
Note that $\Upsilon'={\Psi'}\inv$. We then get
\begin{equation} \label{eq: vector}
\bff^\top=\Upsilon'^{\top}(\Psi_{r+1}^{\top (-1)}-\Psi_{r+1}^{\top}).
\end{equation}

Next, since $\Psi'^{(-1)}=\Phi' \Psi'$, we deduce
	\[ \Theta' \Upsilon'^\top=(\Phi'^{-1})^\top \Upsilon'^{\top}=\Upsilon'^{\top (-1)}. \]
Thus for $n \geq 1$, we have
\begin{equation} \label{eq: log coef}
\Theta'^{(1)} \dots \Theta'^{(n)} \Upsilon'^{\top (n)}=\Upsilon'^{\top}.
\end{equation}

Combining Equations \eqref{eq: vector} and \eqref{eq: log coef}, we get
\begin{align*}
\Theta'^{(1)} \dots \Theta'^{(n)} (\bff^\top)^{(n)} &=\Theta'^{(1)} \dots \Theta'^{(n)} \Upsilon'^{\top (n)}(\Psi_{r+1}^{\top (n-1)}-\Psi_{r+1}^{\top (n)}) \\
&=\Upsilon'^{\top} (\Psi_{r+1}^{\top (n-1)}-\Psi_{r+1}^{\top (n)}) \\
&=\Upsilon'^{\top} F\inv (F\Psi_{r+1}^{\top (n-1)}-F\Psi_{r+1}^{\top (n)}).
\end{align*}
Thus
\begin{align} \label{eq: Stark}
\Log_{\widetilde{E}'}(\bv_{\calM,z}) &= \sum_{n\geq 0} \left (P_n \bv_0\twistk{n}z^n +P_n \bv_1\twistk{n+1} z^{n+1}+P_n \bv_2\twistk{n+2} z^{n+2} + \dots \right )\\
&= \sum_{n\geq 0} \left (P_n \bv_0\twistk{n} +P_{n-1} \bv_1\twistk{n} +P_{n-2} \bv_2\twistk{n}  + \dots \right )z^n\nonumber\\
&= \delta_0   (\bff^\top)+\sum_{n \geq 1} \delta_0   (\Theta'^{(1)} \dots \Theta'^{(n)} (\bff^\top)^{(n)}) z^n \notag \\
&= \delta_0   (\bff^\top)+\sum_{n \geq 1} \delta_0   (\Upsilon'^{\top} F\inv (F\Psi_{r+1}^{\top (n-1)}-F\Psi_{r+1}^{\top (n)})) z^n \notag
\end{align}
where in the second line we consider $P_{i-k}= 0$ if $k>i$.  Since $F \Psi_{r+1}^{\top (k)}$ converges to $0$, it follows that $\Log_{\widetilde{E}'}(\bv_{M,z}) \in \bT_z(\overline K)^d$. 

By evaluating Equation \eqref{eq: Stark} at $z=1$, we obtain a telescoping series on the right-hand side and get
\begin{align} \label{eq:Stark eq}
\Log_{\widetilde{E}'}(\bv_{\calM,z}) \big|_{z=1} &=\delta_0   (\bff^\top)+\sum_{n \geq 1} \delta_0   (\Upsilon'^{\top} (\Psi_{r+1}^{\top (n-1)}-\Psi_{r+1}^{\top (n)}))  \\
&=\delta_0   (\bff^\top)+\delta_0   (\Upsilon'^{\top} \Psi_{r+1}^{\top}) \notag \\
&=\delta_0   (\bff^\top-\Upsilon_{r+1}^\top). \notag
\end{align}
Here the last line holds by \eqref{eq:Upsilon last row}.

We conclude that $\delta_0   (\bff^\top-\Upsilon_{r+1}^\top)$ is a Stark unit for $E'$ and get a Stark logarithmic identity
	\[ \Log^{\text{St}}_{E'}(\bv_{\calM})=\delta_0   (\bff^\top-\Upsilon_{r+1}^\top) \]
which finishes Part (a).  

\subsection{Proof of the main theorem: Part (b)} ${}$\par

In this section we prove Theorem \ref{theorem: main theorem}, Part (b). 

By \eqref{eq:Stark eq} we write
\begin{align*}
\bv_{\calM,z} = \delta_{1,z}   (\bff^\top) &= \delta_{1,z}\left(\sum_i t^{n_i} \sigma^{\ell_i} \left( \sum_{j=1}^d u_{i,j} w_j \right)\right )\\
&=\sum_i \widetilde{E}'_{\theta^{n_i}}  \delta_{1,z} \left(\sigma^{\ell_i} \left( \sum_{j=1}^d u_{i,j} w_j \right)\right).
\end{align*}
Here the last equality follows from the construction of $t$-modules associated to dual $t$-motives as explained in \S \ref{sec: dual motives}. 

We then get the following equality between formal series in $z$:
\begin{align*} 
\Log_{\widetilde{E}'}(\bv_{\calM,z}) &= \sum_i d[\theta^{n_i}] \Log_{\widetilde{E}'}\left(\delta_{1,z} \left(\sigma^{\ell_i} \left( \sum_{j=1}^d u_{i,j} w_j \right)\right) \right) \\
&= \sum_i d[\theta^{n_i}] \Log_{\widetilde{E}'}( (u_{i,1},\ldots,u_{i,d})^\top z^{\ell_i}) \\
&= \sum_i d[\theta^{n_i}] z^{\ell_i} \Log_{\widetilde{E}'}(\bu_i).
\end{align*}
Since all $\bu_i$ are in the domain of convergence of $\Log_{E'}$, Part (a) implies that the above equality holds in the Tate algebra $\bT_z(\bC_\infty)$.

By Part (a) we apply the evaluation map $\ev$ to obtain
\begin{align*}
\delta_0   (\bff^\top-\Upsilon_{r+1}^\top) &= \Log_{\widetilde{E}'}(\bv_{\calM,z}) \big|_{z=1} \\
&=\sum_i d[\theta^{n_i}] z^{\ell_i} \Log_{\widetilde{E}'}(\bu_i) \bigg|_{z=1} \\
&=\sum_i d[\theta^{n_i}] \Log_{E'}(\bu_i), 
\end{align*}
and finishes the proof of Part (b).

\subsection{Relations with Anderson's analytic theory of $A$-finite $t$-modules} ${}$\par

In this section we will apply the elaborate analytic theory of $A$-finite $t$-modules developed by Anderson (see \cite[\S 5.3]{HJ20}) to obtain a result which is similar to Theorem \ref{theorem: main theorem}. A similar analysis appeared in \cite[\S 3.4]{GND20}, which was the starting point of this paper.  

\begin{theorem} \label{theorem: Anderson}
We keep the above notation. Then $\delta_0   (\bff^\top-\Upsilon_{r+1}^\top)$ is a unit for $E'$. Further, we have 
	\[ \Exp_{E'}(\delta_0   (\bff^\top-\Upsilon_{r+1}^\top)) = \bv_{\calM}. \]
\end{theorem}

\begin{remark}
We give some comments to compare Theorems \ref{theorem: main theorem} and \ref{theorem: Anderson}.

\noindent 1) In Theorem \ref{theorem: Anderson} we do not require any restrictions. Consequently, we can only conclude that $\delta_0   (\bff^\top-\Upsilon_{r+1}^\top)$ is a unit, which is much weaker than showing it is a Stark unit as is done in Theorem \ref{theorem: main theorem} (see Remark \ref{rem:units vs Stark units}). Roughly speaking, Theorem \ref{theorem: Anderson} allows us to use the machinery of Stark units and to bypass the convergence issue of logarithm series. This point of view turns out to be very powerful and has already led to several arithmetic applications (for example, compare \cite{ANDTR20a} to \cite{AT90,Pap}, also \cite{ANDTR17a} to \cite{And94,Tae12a}).  

\noindent 2) In addition, we mention again that the proof of Theorem \ref{theorem: Anderson} makes use of Anderson's analytic theory of $A$-finite $t$-modules which is much more complicated than the ingredients given in the proof of Theorem \ref{theorem: main theorem}.
\end{remark}

\begin{proof}[Proof of Theorem \ref{theorem: Anderson}] 
Since $\Psi'^{(-1)}=\Phi' \Psi'$, we have
	\[ \Phi'^\top (({\Psi'}\inv)^\top)\twistinv = ({\Psi'}\inv)^\top.\]
Similarly, since $\Psi^{(-1)}=\Phi \Psi$, we have
	\[\Phi^\top (({\Psi}\inv)^\top)\twistinv = ({\Psi}\inv)^\top.\]
It follows that
	\[ \Phi'^\top (\Upsilon_{r+1}^\top)^{(-1)} + \bff^\top=\Upsilon_{r+1}^\top. \]

Recall that by Anderson's analytic theory of $A$-finite $t$-modules (see \cite[Corollaries 5.20 and 5.21]{HJ20}), if $\bv \in\TT^r$ and $\bz \in \overline K[t]^r$ satisfy
\begin{equation*} 
\Phi'^\top\bv\twistinv - \bv = \bz,
\end{equation*}
then
	\[\Exp_{E'}(\delta_0   (\bv+\bz)) = \delta_1   (\bz) .\]
We apply the above result for $\bv=-\Upsilon_{r+1}^\top$ and $\bz=\bff^\top$ to obtain
	\[\Exp_{E'}(\delta_0   (-\Upsilon_{r+1}^\top+\bff^\top)) = \delta_1   (\bff^\top)= \bv_{\calM} \]
as required.
\end{proof}	


\section{Application to the Anderson-Thakur dual $t$-motives}\label{S:Log for And Thak Models}

\subsection{Some history} ${}$\par

We investigate the Anderson-Thakur dual $t$-motives which were first introduced by Anderson and Thakur in \cite{AT09}.  Shortly thereafter, Chang \cite{Cha14} studied the Anderson-Thakur dual $t$-motives in a general setting and proved many fundamental properties and results. In \cite{CPY19} Chang, Papanikolas and Yu revisited the dual $t$-motives connected to multiple zeta values. They introduced the associated $t$-modules and the corresponding special points and gave an effective criterion for Eulerian MZV's in positive characteristic. Further, Chang, Mishiba and the first author investigated the dual $t$-motives connected to multiple polylogarithms at algebraic points with important applications to $\infty$-adic and $\nu$-adic multiple zeta values in positive characteristic (see  \cite{Cha16,CGM19,CM19,CM19b}).

In this section we apply our main result to obtain log-algebraic identities for the $t$-modules attached to the Anderson-Thakur dual $t$-motives. Then we generalize Chang's theorem in \cite{Cha16} where he gave elegant logarithmic interpretations for special cases of MZV's. We also recover many previously known results in a straightforward way.

\subsection{Anderson-Thakur dual $t$-motives and periods} \label{SS:AndThakModels} ${}$\par

In what follows, let $\fs=(s_1,\dots,s_r) \in \mathbb N^r$ be a tuple for $r \geq 1$ and $\fQ=(Q_1,\ldots,Q_r) \in \overline K[t]^r$ satisfying the condition
\begin{equation} \label{eq: condition for Q}
\lVert Q_r \rVert < |\theta|_\infty ^{\tfrac{s_r q}{q-1}} \text{ and } \lVert Q_i\rVert \leq |\theta|_\infty ^{\tfrac{s_iq}{q-1}} \quad \text{ for all } 1 \leq i \leq r-1.
\end{equation}
We should mention that this condition, inspired by \cite[Remark 4.1.3]{CM19b}, is slightly stronger than that given in \cite[(2.3.1)]{CPY19}, but is enough for applications to multiple zeta values and Carlitz star multiple polylogarithms.

For $1 \leq \ell \leq r$, we set $d_\ell:=s_\ell+\dots+s_r$ and $d:=d_1+\dots+d_r$. Then the Anderson-Thakur dual $t$-motives $\calM'_{\fs,\fQ}$ and $\calM_{\fs,\fQ}$ attached to $\fs$ and $\fQ$ are given by the matrices 
\begin{align*}
\Phi'_{\fs,\fQ} &=
\begin{pmatrix}
(t-\theta)^{s_1+\dots+s_r} & 0 & \dots & 0 \\
Q_1^{(-1)} (t-\theta)^{s_1+\dots+s_r} & (t-\theta)^{s_2+\dots+s_r} & \dots & 0  \\
0 & Q_2^{(-1)} (t-\theta)^{s_2+\dots+s_r} & & 0 \\
\vdots & & \ddots & \vdots \\
0 & 0 & \dots & (t-\theta)^{s_r}
\end{pmatrix} \in \Mat_r(\overline K[t]), \\
\Phi_{\fs,\fQ} &=
\begin{pmatrix}
(t-\theta)^{s_1+\dots+s_r} & 0 & 0 & \dots & 0 \\
Q_1^{(-1)} (t-\theta)^{s_1+\dots+s_r} & (t-\theta)^{s_2+\dots+s_r} & 0 & \dots & 0 \\
0 & Q_2^{(-1)} (t-\theta)^{s_2+\dots+s_r} & \ddots & & \vdots \\
\vdots & & \ddots & (t-\theta)^{s_r} & 0 \\
0 & \dots & 0 & Q_r^{(-1)} (t-\theta)^{s_r} & 1 
\end{pmatrix} \in \Mat_{r+1}(\overline K[t]).
\end{align*}
From now on, to simplify the notation, we will drop the subscripts $\fs$ and $\fQ$ whenever no confusion results. For example, we will write $\Phi$ instead of $\Phi_{\fs,\fQ}$ and so on.


Throughout this paper, we work with the Carlitz period $\widetilde \pi$ which is a fundamental period of the Carlitz module (see \cite{Gos96, Tha04}). We make a choice of $(q-1)$st root of $(-\theta)$ and set 
	\[ \Omega(t):=(-\theta)^{-q/(q-1)} \prod_{i \geq 1} \left(1-\frac{t}{\theta^{q^i}} \right) \in \bT^\times \]
so that $\Omega^{(-1)}=(t-\theta)\Omega$ and
\begin{equation} \label{eq:value of omega}
\frac{1}{\Omega(\theta)}=\widetilde \pi.
\end{equation}

Given $\fs$ and $\fQ$ satisfying \eqref{eq: condition for Q} as above, Chang introduced the following series (see \cite[Lemma 5.3.1]{Cha14} and also \cite[Equation (2.3.2)]{CPY19}):	
\begin{align} \label{eq: series L}
\mathfrak L_{\fs,\fQ}:=\sum_{i_1 > \dots > i_r \geq 0} (\Omega^{s_r} Q_r)^{(i_r)} \dots (\Omega^{s_1} Q_1)^{(i_1)}. 
\end{align}
We also need the star series
\begin{align} \label{eq: series L star}
\mathfrak L^*_{\fs,\fQ}:=\sum_{i_1 \geq \dots \geq i_r \geq 0} (\Omega^{s_r} Q_r)^{(i_r)} \dots (\Omega^{s_1} Q_1)^{(i_1)}.
\end{align}

If we denote $\mathcal E$ the ring of series $\sum_{n \geq 0} a_nt^n \in \overline K[[t]]$ such that $\lim_{n \to +\infty} \sqrt[n]{|a_n|_\infty}=0$ and $[K_\infty(a_0,a_1,\ldots):K_\infty]<\infty$, then any $f \in \mathcal E$ is an entire function. It is proved that $\mathfrak L_{\fs,\fQ} \in \mathcal E$ (see \cite[Lemma 5.3.1]{Cha14}).

More generally, for $1 \leq \ell < j \leq r+1$, we define the series
\begin{align} \label{eq: series L general}
\mathfrak L(s_\ell,\dots,s_{j-1}) & :=\sum_{i_\ell > \dots > i_{j-1} \geq 0} (\Omega^{s_{j-1}} Q_{j-1})^{(i_{j-1})} \dots (\Omega^{s_\ell} Q_\ell)^{(i_\ell)}, \\ \mathfrak L^*(s_\ell,\dots,s_{j-1}) & :=\sum_{i_\ell \geq \dots \geq i_{j-1} \geq 0} (\Omega^{s_{j-1}} Q_{j-1})^{(i_{j-1})} \dots (\Omega^{s_\ell} Q_\ell)^{(i_\ell)}, \notag
\end{align}
which are the series in \eqref{eq: series L} and \eqref{eq: series L star} attached to $(s_\ell,\dots,s_{j-1})$ and $(Q_\ell,\dots,Q_{j-1})$. We should mention that we omit the subscript $\fQ$ from the definition of the above series to avoid heavy notation.

\begin{lemma}\label{L:Lstarinv}
For $1 \leq \ell \leq j \leq r$, we have 
\begin{align*}
(-1)^\ell \mathfrak L^*(s_j,\dots,s_\ell) = \sum_{k=\ell+1}^j (-1)^{k-1} \mathfrak L(s_\ell,\dots,s_{k-1}) \mathfrak L^*(s_j,\dots,s_k) + (-1)^j \mathfrak L(s_\ell,\dots,s_j).
\end{align*}
and 
\begin{align*}
(-1)^j \mathfrak L^*(s_j,\dots,s_\ell) = \sum_{k=\ell+1}^j (-1)^k \mathfrak L(s_k\dots,s_j) \mathfrak L^*(s_{k-1},\dots,s_\ell) + (-1)^\ell \mathfrak L(s_\ell,\dots,s_j).
\end{align*}
\end{lemma}
\begin{proof}
The follows similarly to the proof of \cite[4.2.1]{CM19} and is a straightforward exercise in the inclusion/exclusion principal.  We leave the details to the reader.
\end{proof}

The matrix given by	
\begin{align*}
\Psi &=
\begin{pmatrix}
\Omega^{s_1+\dots+s_r} & 0 & 0 & \dots & 0 \\
\mathfrak L(s_1) \Omega^{s_2+\dots+s_r} & \Omega^{s_2+\dots+s_r} & 0 & \dots & 0 \\
\vdots & \mathfrak L(s_2) \Omega^{s_3+\dots+s_r} & \ddots & & \vdots \\
\vdots & & \ddots & \ddots & \vdots \\
\mathfrak L(s_1,\dots,s_{r-1}) \Omega^{s_r}  & \mathfrak L(s_2,\dots,s_{r-1}) \Omega^{s_r} & \dots & \Omega^{s_r}& 0 \\
\mathfrak L(s_1,\dots,s_r)  & \mathfrak L(s_2,\dots,s_r) & \dots & \mathfrak L(s_r) & 1 
\end{pmatrix} \in \text{GL}_{r+1}(\bT)
\end{align*}
satisfies 
	\[ \Psi^{(-1)}=\Phi \Psi. \]
Thus $\Psi$ is a rigid analytic trivialization associated to the dual $t$-motive $\calM$. 

Using Lemma \ref{L:Lstarinv} we see that the periods of $\calM$ are given by the matrix $\Upsilon=\Psi\inv$:
\begin{align*}
\Upsilon =
\begin{pmatrix}
\Omega^{-(s_1+\dots+s_r)} & 0 & 0 & \dots & 0 \\
-\mathfrak L^*(s_1) \Omega^{-(s_1+\dots+s_r)} & \Omega^{-(s_2+\dots+s_r)} & 0 & \dots & 0 \\
\vdots & -\mathfrak L^*(s_2) \Omega^{-(s_2+\dots+s_r)} & \ddots & & \vdots \\
\vdots & & \ddots & \ddots & \vdots \\
(-1)^{r-1} \mathfrak L^*(s_{r-1},\dots,s_1) \Omega^{-(s_1+\dots+s_r)} & (-1)^{r-2} \mathfrak L^*(s_{r-1},\dots,s_2) \Omega^{-(s_2+\dots+s_r)} & \dots & \Omega^{-s_r}& 0 \\
(-1)^r \mathfrak L^*(s_r,\dots,s_1) \Omega^{-(s_1+\dots+s_r)} & (-1)^{r-1} \mathfrak L^*(s_r,\dots,s_2) \Omega^{-(s_2+\dots+s_r)} & \dots & -\mathfrak L^*(s_r) \Omega^{-s_r} & 1 
\end{pmatrix}.
\end{align*}
Note that $\Upsilon \in \text{GL}_{r+1}(\bT)$.

\begin{lemma} \label{lemma: equality star}
For $1 \leq \ell \leq j \leq r$, we have
\begin{align*}
& \mathfrak L^*(s_{j-1},\dots,s_\ell)^{(-1)} \\
&=\mathfrak L^*(s_{j-1},\dots,s_\ell)+\mathfrak L^*(s_{j-1},\dots,s_{\ell+1})Q_\ell^{(-1)} [(t-\theta) \Omega]^{s_\ell}+\dots+Q_\ell^{(-1)} \dots Q_{j-1}^{(-1)} [(t-\theta) \Omega]^{(s_\ell+\dots+s_{j-1})}.
\end{align*}
\end{lemma}

\begin{proof}
Since $\Psi^{(-1)}=\Phi \Psi$, we get $\Upsilon^{(-1)}=\Upsilon \Phi^{-1}$. Using the above formulas we deduce the required equality by direct calculations. We leave the details to the reader.
\end{proof}

\subsection{The associated $t$-modules and special points} \label{SS:AndThak_modules} ${}$\par

We now present the associated $t$-modules and special points inspired by the work of Chang, Papanikolas and Yu \cite{CPY19}. Let $\bm=\{m_1,\dots,m_r\}$ be the $\overline K[t]$-basis of $\calM'$ with respect to the action of $\sigma$ represented by $\Phi'$. It is not hard to check that $\calM'$ is a free left $\overline K[\sigma]$-module of rank $d=(s_1+\dots+s_r)+(s_2+\dots+s_r)+ \dots + s_r$ and that 
\begin{equation}\label{D:Ksigmabasis}
\bw:=\{w_1,\ldots,w_d\}:=\{(t-\theta)^{s_1+\dots+s_r-1} m_1,\dots,m_1,\dots,(t-\theta)^{s_r-1} m_r,\dots,m_r \} 
\end{equation}
is a $\overline K[\sigma]$-basis of $\calM'$. We further observe that $(t-\theta)^\ell \calM'/\sigma \calM' =(0)$ for $\ell \gg 0$.

For such $\calM'$, we recall that we can identify $\calM'/(\sigma-1)\calM'$ with the direct sum of $d$ copies of $\overline K$ as follows. Fixing a $\overline K[\sigma]$-basis $\bw=\{w_1,\dots,w_d\}$ of $\calM'$ given as above, we can express any $m \in \calM'$ as
	\[ m=\sum_{i=1}^d u_i w_i, \quad u_i \in \overline K[\sigma], \]
and then can write down $\delta_1:\calM' \to \Mat_{d \times 1}(\overline K)$ from \S \ref{sec: dual motives} by
	\[ \delta_1(m):=(\delta(u_1),\dots,\delta(u_d))^\top=\begin{pmatrix}
\delta(u_1) \\
\vdots \\
\delta(u_d)
\end{pmatrix} \]
where 
	\[ \delta\left (\sum_i c_i \sigma^i\right ) = \sum_i c_i^{q^i}. \]
It follows that $\delta_1$ is a map of $\Fq$-vector spaces with kernel $(\sigma-1) \calM'$. We note that if $(b_1,\dots,b_d)^\top \in \Mat_{d \times 1}(\overline K)$, then there is a natural lift to $\calM'$, since 
	\[ \delta_1(b_1 w_1+\dots+b_d w_d)=(b_1,\dots,b_d)^\top. \]
We denote by $E'$ the Anderson $t$-module defined over $\overline K$ with $E'(\overline K)$ identified with $\Mat_{d \times 1}(\overline K)$ on which the $\Fq[t]$-module structure given by 
	\[ E': \Fq[\theta] \to \Mat_d(\overline K)[\tau] \]
so that 
	\[ \delta_1(t(b_1 w_1+\dots+b_d w_d))=E'_\theta(b_1,\dots,b_d)^\top=E'_\theta \begin{pmatrix}
b_1 \\
\vdots \\
b_d
\end{pmatrix}. \]
Then $E'$ is the $t$-module associated the dual $t$-motive $\calM'$ as explained in \S \ref{sec: dual motives}. 

We can write down explicitly the map $\delta_0:\calM' \to \Mat_{d \times 1}(\overline K)$. Let $m \in \calM'=\overline K[t]m_1+\dots+\overline K[t]m_r$. Then we can write
	\[ m=\sum_{\ell=1}^r (c_{d_\ell-1,\ell} (t-\theta)^{d_\ell-1}+\dots+c_{0,\ell}+F_\ell(t) (t-\theta)^{d_\ell}) m_\ell, \]
with $c_{i,\ell} \in \overline K$ and $F_\ell(t) \in \overline K[t]$. Then
\begin{equation} \label{D:delta0ATmodel}
\delta_0(m):=(c_{d_1-1,1},\dots,c_{0,1},\dots,c_{d_r-1,r},\dots,c_{0,r})^\top. 
\end{equation}

Inspired by Chang-Papanikolas-Yu (see \cite[\S 5.3]{CPY19}) we define the point 
\begin{equation} \label{eq: CPY point}
\bv_{\fs,\fQ}:=\bv_{\calM} := \delta_1(Q_r^{(-1)}(t-\theta)^{s_r} m_r) \in E'(\overline K).
\end{equation}

\subsection{Logarithm series} ${}$\par

The coefficients of the logarithm series can be calculated following \cite{ANDTR20a}. In this particular case, it was also done in \cite[\S 4.2]{CGM19}.

We set
\begin{align*}
\Theta=(\Phi^{-1})^\top &=
\begin{pmatrix}
\Theta_{1,1} & \dots & \Theta_{1,r+1}  \\
& \ddots & \vdots \\
& & \Theta_{r+1,r+1}
\end{pmatrix} \in \Mat_{r+1}(\overline K(t))
\end{align*}
where for $1 \leq i \leq j \leq r+1$, 
	\[ \Theta_{i,j}=(-1)^{j-i} \frac{\prod_{i \leq k <j} Q_k^{(-1)}}{(t-\theta)^{s_j+\dots+s_r}}. \]
We set
\begin{align*}
\Theta':=((\Phi')^{-1})^\top &=
\begin{pmatrix}
\Theta_{1,1} & \dots & \Theta_{1,r}  \\
& \ddots & \vdots \\
& & \Theta_{r,r}
\end{pmatrix} \in \Mat_r(\overline K(t)).
\end{align*}
If we write 
	\[ \Log_{E'}=\sum_{n \geq 0} P_n \tau^n, \]
then by \cite[Proposition 2.2]{ANDTR20a}, for $n \geq 0$, the $n$th coefficient of the logarithm series of $E'$ evaluated at $\bv \in \overline K^d$ is given by 	
	\[ P_n \bv\twistk{n} = \delta_0(\Theta'^{(1)} \dots \Theta'^{(n)} \iota\inv (\bv)\twistk{n}). \]

\subsection{Log-algebraic identities for $t$-modules associated to Anderson-Thakur dual $t$-motives} ${}$\par

In this section we apply Theorem \ref{theorem: main theorem} to obtain log-algebraic identities for the $t$-module $E'$ associated to the Anderson-Thakur dual $t$-motive $\calM'$.

\begin{theorem} \label{theorem: AT model}
We have a split-logarithmic identity
\begin{align*}
\Log_{E'}^!(\bv_{\fs,\fQ})
=\delta_0 \begin{pmatrix}
(-1)^{r-1} \mathfrak L^*(s_r,\dots,s_1) \Omega^{-(s_1+\dots+s_r)} \\
(-1)^{r-2} \mathfrak L^*(s_r,\dots,s_2) \Omega^{-(s_2+\dots+s_r)} \\
\vdots \\
\mathfrak L^*(s_r) \Omega^{-s_r} 
\end{pmatrix}.
\end{align*}
\end{theorem}

\begin{proof}
Note that 
\begin{align*}
\bff &=(0,\dots,0,Q_r^{(-1)} (t-\theta)^{s_r}), \\
\Psi_{r+1} &=(\mathfrak L(s_1,\dots,s_r),\mathfrak L(s_2,\dots,s_r),\dots,\mathfrak L(s_r)).
\end{align*}
So $\Psi_{r+1}\twistk{k} \to 0$ as $k\to \infty$ by \cite[Lemma 5.3.1]{Cha14}. Recall that $\bv_{\fs,\fQ}=\delta_1(Q_r^{(-1)} (t-\theta)^{s_r} m_r)$, and we need to express $Q_r^{(-1)} (t-\theta)^{s_r} m_r$ in the $\overline K[\sigma]$-basis
	\[ \bw=\{(t-\theta)^{s_1+\dots+s_r-1} m_1,\dots,(t-\theta) m_1, m_1,\dots,(t-\theta)^{s_r-1} m_r,\dots,(t-\theta) m_r,m_r \} \] 
of $\calM'$ (see \eqref{D:Ksigmabasis}). By definition we have $\sigma m_1=(t-\theta)^{s_1+\dots+s_r} m_1$ and for $1 < \ell \leq r$, 
\[\sigma m_\ell=Q_{\ell-1}^{(-1)} (t-\theta)^{s_{\ell-1}+\dots+s_r} m_{\ell-1}+ (t-\theta)^{s_\ell+\dots+s_r} m_\ell.\]
It follows that
\begin{align*}
Q_r^{(-1)} (t-\theta)^{s_r} m_r &= Q_r^{(-1)}(\sigma m_r-Q_{r-1}^{(-1)} (t-\theta)^{s_{r-1}+s_r} m_{r-1}) \\
&= Q_r^{(-1)} \sigma m_r - Q_r^{(-1)} Q_{r-1}^{(-1)}(\sigma m_{r-1} - Q_{r-2}^{(-1)} (t-\theta)^{s_{r-2}+\dots+s_r} m_{r-2}) \\
&= \dots \\
&= \sum_{\ell=1}^r (-1)^{r-\ell} Q_r^{(-1)} \dots Q_\ell^{(-1)} \sigma m_\ell.
\end{align*}
Then we write
\begin{align*}
Q_r\cdots Q_\ell &= (b_{r,0} + b_{r,1}t + \dots + b_{r,m_r}t^{m_r}) \ldots (b_{\ell,0} + b_{\ell,1}t + \dots + b_{\ell,m_\ell}t^{m_\ell})\\
&=\sum_{\substack{\{i_\ell,\dots,i_r\}\in \prod_{\ell\leq j\leq r}[0,\dots,m_j]}} (b_{r,i_r}\cdot b_{r-1,i_{r-1}}\cdots b_{\ell,i_\ell})t^{i_r + \dots + i_\ell}.
\end{align*}
We then substitute the above expression into the preceding expression for $Q_r^{(-1)} (t-\theta)^{s_r} m_r$ to obtain an expression of the form
\[Q_r^{(-1)} (t-\theta)^{s_r} m_r = \sum_i t^{n_i} \sigma^{\ell_i} \left( \sum_{j=1}^d u_{i,j} w_j \right),\]
for triples $(\ell_i,n_i,\bu_i=(u_{i,1},\ldots,u_{i,d})^\top)\in \mathbb Z^{\geq 0} \times \mathbb Z^{\geq 0} \times \C_\infty^d$, where $i$ is indexed over some finite set.  As the coefficients $b_{i,j}$ are coefficients of the polynomials $Q_i$, by \eqref{eq: condition for Q}, we know that $\lVert Q_r \rVert < |\theta|_\infty ^{\tfrac{s_r q}{q-1}}$ and $\lVert Q_i \rVert \leq |\theta|_\infty ^{\tfrac{s_iq}{q-1}}$ for $1 \leq i \leq r-1$. Therefore
\[|b_{r,i_r}\cdot b_{r-1,i_{r-1}}\cdots b_{\ell,i_\ell}|_\infty < |\theta|_\infty ^{\tfrac{(s_r + \dots + s_\ell)q}{q-1}}.\]
Then by \cite[Lemma 4.2.1]{CGM19} each $\bu_{i}$ is inside the radius of convergence of $\Log_{E'}$.  Thus the $t$-module $E'$, the point $\Psi_{r+1}$ and the point $v_{\fs,\fQ}$ satisfy the conditions of Theorem \ref{theorem: main theorem} (b), which we apply.  The final observation is that by the above calculations, $Q_r^{(-1)} (t-\theta)^{s_r} m_r \in \sigma(\calM')$ and hence $\delta_0(\bff^\top) = 0$, which allows us to apply the last statement of Theorem \ref{theorem: main theorem} and finishes the proof.
\end{proof}

\subsection{Relations with a theorem of Chang-Papanikolas-Yu} ${}$\par

We now apply Theorem \ref{theorem: CPY model} to obtain another proof of \cite[Theorem 2.5.2 (a) ($\Leftarrow$) and (b)]{CPY19} in our setting. 

\begin{proposition} \label{proposition: torsion}
Suppose that $\calM$ represents a torsion class in $\textrm{Ext}^1_\cF(\mathbf{1},\calM')$. Then all the values $\mathfrak L(s_1,\dots,s_r)(\theta),\dots,\mathfrak L(s_r)(\theta)$ are in $K$.
\end{proposition}

\begin{proof}
By Remark \ref{rem:torsion class}, $\bv_{\fs,\fQ}$ is a torsion point in $E'(\overline K)$ since $\calM$ represents a torsion class in $\textrm{Ext}^1_\cF(\mathbf{1},\calM')$. It follows that $d[a] \Log_{E'}^!(\bv_{\fs,\fQ})$ is a period of $E'$ for some $a \in A$. Thus we can write 
\begin{equation} \label{eq: torsion}
\Log_{E'}^!(\bv_{\fs,\fQ})=d[a_1] \lambda_1 + d[a_2] \lambda_2 + \dots + d[a_r] \lambda_r, \quad a_i \in K,
\end{equation}
where $\lambda_i$ are the $A$-basis of the period lattice $\Lambda_{E'}$ given by the map $\delta_0$ applied to the column vectors of $\Upsilon'^\top$ (see \cite[Cor. 5.21]{HJ20} and also \cite[Lemma 3.7]{GND20}). 

For $1 \leq \ell \leq r$, we consider the $(d_1+\dots+d_\ell)$th coordinate of both sides in \eqref{eq: torsion}. Then
\begin{enumerate}
\item By Theorem \ref{theorem: CPY model} and \eqref{eq:value of omega}, the $(d_1+\dots+d_\ell)$th coordinate of $\Log_{E'}^!(\bv_{\fs,\fQ})$ equals $(-1)^{r-\ell} \mathfrak L^*(s_r,\dots,s_\ell)(\theta) \cdot \widetilde{\pi}^{s_\ell+\dots+s_r}$.

\item For $1 \leq j < \ell$, the $(d_1+\dots+d_\ell)$th coordinate of $\lambda_j$ is zero.

\item The $(d_1+\dots+d_\ell)$th coordinate of $\lambda_\ell$ equals $\widetilde{\pi}^{s_\ell+\dots+s_r}$.

\item The matrices $d[a_i]$ are upper triangular and equal $a_i$ along the main diagonal.  This can be seen quickly from the definition of the $\overline K[\sigma]$-basis \eqref{D:Ksigmabasis} and from the definition of $\delta_0$ in \eqref{D:delta0ATmodel}. 
\end{enumerate}

Thus, by descending induction on $1 \leq \ell \leq r$, we use Lemma \ref{L:Lstarinv} to get
	\[ a_\ell=\mathfrak L(s_\ell,\dots,s_r)(\theta). \]
Since $a_\ell \in K$ for all $1 \leq \ell \leq r$, we deduce that all the values 
	\[ \mathfrak L(s_1,\dots,s_r)(\theta),\mathfrak L(s_2,\dots,s_r)(\theta),\dots,\mathfrak L(s_r)(\theta) \] 
are in $K$.
\end{proof}

\begin{remark} \label{rem:ABP extension}
We explain briefly how to extend the above result to the more general setting considered in \cite{CPY19} and leave the interested reader to fill in the details. 

We put $w:=\sum_{i=1}^r s_i$ and let $Q \in \overline K[t]$ such that $\lVert Q \rVert<|\theta|^{wq/(q-1)}$. We consider the effective dual $t$-motive $\mathcal N \in \textrm{Ext}^1_\cF(\mathbf{1},\calM')$ defined by the matrix
\begin{align*}
\begin{pmatrix}
\Phi' & 0  \\
\mathbf{u}_{w,Q} & 1 
\end{pmatrix} \in \Mat_{r+1}(\overline K[t]), 
\end{align*}
with $\mathbf{u}_{w,Q}=(Q\twist(t-\theta)^w,0,\dots,0)\in \Mat_{1\times r}(\overline K[t])$. Note that $\mathcal N$ admits a rigid analytic trivialization given by 
\begin{align*}
\begin{pmatrix}
\Psi' & 0  \\
(\mathfrak L_{w,Q},0,\dots,0) & 1 
\end{pmatrix} \in \Mat_{r+1}(\overline K[t]), 
\end{align*}
where $\mathfrak L_{w,Q}$ is the series in \eqref{eq: series L} attached to $(w)$ and $(Q)$.

We apply our method to obtain log-algebraic identities for the $t$-module attached to $\mathcal N$. Consequently, we get \cite[Theorem 2.5.2 (a) ($\Leftarrow$) and (b)]{CPY19} which states that if the classes of $\calM$ and $\mathcal N$ are $\Fq[t]$-linearly dependent in $\textrm{Ext}^1_\cF(\mathbf{1},\calM')$, then all the values $\mathfrak L(s_2,\dots,s_r)(\theta),\dots,\mathfrak L(s_r)(\theta)$ are in $K$.
\end{remark}

\begin{remark}  \label{rem: ABPconverse}
We should mention that by using the powerful ABP criterion \cite{ABP04} and also \cite{Pap08}, the converse was also proved in \cite[Theorem 2.5.2(a) ($\Rightarrow$)]{CPY19} under the mild conditions that the values $\mathfrak L(s_\ell,\dots,s_{j-1})(\theta)$ do not vanish for $1 \leq \ell<j\leq r+1$.
\end{remark}

\subsection{A generalization of a theorem of Chang} \label{sec: Chang} ${}$\par

In the fundamental work \cite{AT90} Anderson and Thakur gave logarithmic interpretations for Carlitz zeta values, i.e. depth-one multiple zeta values. In \cite{Cha16} Chang presented very simple and elegant logarithmic interpretations for some special MZV's (see \cite[Theorem 4.1.1]{Cha16}) and deduced an effective criterion for the dimension of depth-two multiple zeta values. However, as Chang and Mishiba \cite{CM20} explained to us, to their knowledge, the relations among Chang's theorem and the works of Chang-Papanikolas-Yu \cite{CPY19} and Chang-Mishiba \cite{CM19,CM19b} are still mysterious.

The aim of this section is to present a generalization of Chang's theorem as an application of our main result (see Theorem \ref{theorem: generalization Chang}). As a consequence, we clarify the connection between the work of Chang \cite{Cha16} and that of Chang-Papanikolas-Yu \cite{CPY19}. We close this section by deducing an unusual formula of Thakur from Chang's theorem (see Remark \ref{rem: strange formula}).


\begin{theorem} \label{theorem: generalization Chang}
Let $\fs=(s_1,\ldots,s_r) \in \N^r$ with $r \geq 2$. Assume that, for $1 \leq \ell<j\leq r+1$, the values $\mathfrak L(s_\ell,\dots,s_{j-1})(\theta)$ do not vanish. We further suppose that $\mathfrak L(s_2,\ldots,s_r)(\theta) \in K$. Then there exist $a_\fs \in A$, an integral point $\mathbf{Z}_\fs \in C^{\otimes (s_1+\ldots+s_r)}(A)$ and a point $\bz_\fs \in \C_\infty^{s_1+\ldots+s_r}$ such that 

1) the last coordinate of $\bz_\fs$ equals $a_\fs \mathfrak L(s_1,\ldots,s_r)(\theta)$,

2) $\Exp_{C^{\otimes (s_1+\ldots+s_r)}}(\bz_\fs)=\mathbf{Z}_\fs$.
\end{theorem} 

\begin{proof}
Since the values $\mathfrak L(s_\ell,\dots,s_{j-1})(\theta)$ do not vanish  for $1 \leq \ell<j\leq r+1$, the hypothesis of \cite[Theorem 2.5.2]{CPY19} holds. Thus this theorem implies that $\mathfrak L(s_3,\ldots,s_r)(\theta) ,\ldots,\mathfrak L(s_r)(\theta)$ are also in $K$ since $\mathfrak L(s_2,\ldots,s_r)(\theta) \in K$. 

For $2 \leq \ell \leq r$, we set
	\[ a_\ell=\mathfrak L(s_\ell,\dots,s_r) \in K. \]
We take $a_\fs \in A$ such that $a_\fs a_\ell \in A$ for all $2 \leq \ell \leq r$.

We denote by $\lambda_1,\ldots,\lambda_r$ the $A$-basis of the period lattice $\Lambda_{E'}$ given by the map $\delta_0$ applied to the column vectors of $\Upsilon'^\top$ (see \cite[Cor. 5.21]{HJ20} and also \cite[Lemma 3.7]{GND20}). 

For $1 \leq \ell \leq r$, we consider the $(d_1+\dots+d_\ell)$th coordinate of $\Log_{E'}^!(\bv_{\fs,\fQ})$ and $\lambda_1,\ldots,\lambda_r$. Then
\begin{enumerate}
\item By Theorem \ref{theorem: CPY model} and \eqref{eq:value of omega}, the $(d_1+\dots+d_\ell)$th coordinate of $\Log_{E'}^!(\bv_{\fs,\fQ})$ equals $(-1)^{r-\ell} \mathfrak L^*(s_r,\dots,s_\ell)(\theta) \cdot \widetilde{\pi}^{s_\ell+\dots+s_r}$.

\item For $1 \leq j < \ell$, the $(d_1+\dots+d_\ell)$th coordinate of $\lambda_j$ is zero.

\item The $(d_1+\dots+d_\ell)$th coordinate of $\lambda_\ell$ equals $\widetilde{\pi}^{s_\ell+\dots+s_r}$.

\item The matrices $d[a_i]$ are upper triangular and equal $a_i$ along the main diagonal.  This can be seen quickly from the definition of the $\overline K[\sigma]$-basis \eqref{D:Ksigmabasis} and from the definition of $\delta_0$ in \eqref{D:delta0ATmodel}. 
\end{enumerate}

We consider
\[ \bz_\fs'=d[a_\fs]\Log_{E'}^!(\bv_{\fs,\fQ})-d[a_\fs a_2] \lambda_2- \dots -d[a_\fs a_r] \lambda_r. \]
Then we deduce 
\begin{enumerate}
\item The $d_1$th coordinate of $\bz_\fs'$ equals $a_\fs \mathfrak L(s_\ell,\dots,s_r)(\theta)$ by Lemma \ref{L:Lstarinv}.

\item For $d_1 < j \leq d_1+\dots+d_r$, the $j$th coordinate of $\bz_\fs'$ is zero by \cite[Theorem 2.5.2]{CPY19} (see also Proposition \ref{proposition: torsion}, Remark \ref{rem: ABPconverse} and \cite[Lemma 3.4.5]{CGM19}).
\end{enumerate}
Thus to conclude it suffices to choose $\bz_\fs$ to be the first $d_1$ coordinates of $\bz_\fs'$. This finishes the proof.
\end{proof}

\begin{remark}
1) The proof presented above grew out of many discussions of the second author and F. Pellarin to whom he would like to express his gratitude.

2) Chang \cite{CM20} informed us that Yen-Tsung Chen and Harada are working on generalizing Chang's result to the case where $\fQ=(u_1,\dots,u_r) \in \overline K^r$ satisfying $|u_r|_{\infty} < q^{\frac{s_r q}{q-1}}$ and $|u_i|_{\infty} \leq q^{\frac{s_i q}{q-1}}$ for $1 \leq i \leq r-1$.
\end{remark}

\begin{remark} \label{rem: strange formula}
If we write the Carlitz logarithm attached to the Carlitz module $C$ as
	\[ \log_C=\sum_{i \geq 0} \frac{1}{\ell_i} \tau^i, \quad \ell_i \in A, \]
then in \cite[Theorem 6]{Tha09} Thakur gave the following ``strange" formula
\begin{equation} \label{eq: curious id}
\zeta_A(1,q^3-1)=\left( \frac{1}{\ell_3}+\frac{1}{\ell_2}+\frac{\theta}{\ell_2} \right) \zeta_A(q^3) - \frac{1}{\ell_2} \left(\log_C(\theta^{1/q}) \right)^{q^3}.
\end{equation}

We claim that this identity can be seen as an explicit example of the above Theorem. In fact, we put $\fs=(1,q^3-1)$ and consider the tensor power $C^{\otimes q^3}$. We know that, by \cite{AT90}, the last row of the logarithm associated to $C^{\otimes q^3}$ denoted by $\iota\inv(\Log_{C^{\otimes q^3}})$ is given by
	\[ \iota\inv(\Log_{C^{\otimes q^3}}(0,\ldots,0,x)^\top)=\sum_{i \geq 0} \frac{1}{\ell_i^{q^3}} \tau^i(x). \]
Thus
	\[ \left(\log_C(\theta^{1/q}) \right)^{q^3}=\sum_{i \geq 0} \frac{1}{\ell_i^{q^3}} \tau^{i+3}(\theta^{1/q})=\sum_{i \geq 0} \frac{1}{\ell_i^{q^3}} \tau^i(\theta^{q^2})=\iota\inv(\Log_{C^{\otimes q^3}}(0,\ldots,0,\theta^{q^2})^\top). \]
The celebrated Anderson-Thakur theorem \cite[Theorem 3.8.3]{AT90} shows that $\zeta_A(q^3)$ can be interpreted as the last coordinate of $\Log_{C^{\otimes q^3}}$. We conclude that \eqref{eq: curious id} gives an explicit interpretation for the MZV $\zeta_A(1,q^3-1)$ as the last coordinate of a split-logarithmic identity involving $\Log_{C^{\otimes q^3}}$ as is implied by Chang's theorem.

F. Pellarin has informed us that, in an ongoing project with O. Gezmis, they construct more examples of such explicit identities for MZV's.
\end{remark}

\subsection{Log-algebraic identities for Chang-Papanikolas-Yu's $t$-modules} \label{sec: CPY} ${}$\par

In this section we specialize $\fQ=(Q_1,\ldots,Q_r)$ to Anderson-Thakur polynomials and study the corresponding $t$-modules considered in the work of Chang, Papanikolas and Yu \cite{CPY19} (see also \cite{AT09}). Then we apply Theorem \ref{theorem: main theorem} to obtain several applications to this case.

These dual $t$-motives are related to the multiple zeta values defined by Thakur \cite{Tha04} as follows. For any tuple of positive integers $\fs=(s_1,\ldots,s_r) \in \N^r$, we introduce
\begin{equation*} 
\zeta_A(\fs)=\zeta_A(s_1,\ldots,s_r):=\sum \frac{1}{a_1^{s_1} \ldots a_r^{s_r}} \in K_\infty
\end{equation*}
where the sum runs through the set of tuples $(a_1,\ldots,a_r) \in A_+^r$ with $\deg a_1>\ldots>\deg a_r$; $r$ is called the depth and $w:=s_1+\ldots+s_r$ the weight of $\zeta_A(\fs)$. Depth one MZV's are also called Carlitz zeta values (see \cite{Car35}). It is proved that $\zeta_A(\fs)$ are nonzero by Thakur \cite{Tha09}. We refer the reader to the excellent surveys \cite{Tha17,Tha20} for more details about MZV's.

We briefly review Anderson-Thakur polynomials introduced in \cite{AT90}. For $k \geq 0$, we set
\begin{align*}
[k] &:=\theta^{q^k}-\theta, \\
D_k &:= \prod^k_{\ell=1} [\ell]^{q^{k-\ell}}=[k][k-1]^q \dots [1]^{q^{k-1}}.
\end{align*}
For $n \in \N$, we write
	\[ n-1 = \sum_{j \geq 0} n_j q^j, \quad 0 \leq n_j \leq q-1, \]
and define
	\[ \Gamma_n:=\prod_{j \geq 0} D_j^{n_j}. \]
We set 
\begin{align*}
\gamma_0(t) &:=1, \\
\gamma_j(t) &:=\prod^j_{\ell=1} (\theta^{q^j}-t^{q^\ell}),\quad j\geq 1.
\end{align*}
Then Anderson-Thakur polynomials $\alpha_n(t) \in A[t]$ are given by the generating series
	\[ \sum_{n \geq 1} \frac{\alpha_n(t)}{\Gamma_n} x^n:=x\left( 1-\sum_{j \geq 0} \frac{\gamma_j(t)}{D_j} x^{q^j} \right)^{-1}. \]
Finally, we define $H_n(t)$ by switching $\theta$ and $t$:
	\[ H_n(t)=\alpha_n(t) \big|_{t=\theta,\theta=t}.\]
By \cite[3.7.3]{AT90} we get that $\lVert H_n \rVert < |\theta|_\infty ^{\tfrac{n q}{q-1}}.$ Thus the polynomials $(Q_1,\ldots,Q_r)=(H_{s_1},\ldots,H_{s_r})$ satisfy \eqref{eq: condition for Q}. 

In what follows, we will specialize the $t$-motives $\calM$ and $\calM'$ from the previous sections to $(Q_1,\ldots,Q_r)=(H_{s_1},\ldots,H_{s_r})$ and get  logarithmic interpretations for multiple zeta star values. 

We wish to study the point $\bv_\fs \in E'(\overline K)$ which corresponds to $H_{s_r}^{(-1)}(t-\theta)^{s_r} m_r \in \calM'/(\sigma-1)\calM'$. Note that this point was first introduced by Chang, Papanikolas and Yu in \cite{CPY19} and played an important role in their effective criterion to determine whether the corresponding multiple zeta value $\zeta_A(\fs)$ is Eulerian. Further, they proved the following integrality result:
\begin{theorem}[\cite{CPY19}, Theorem 5.3.4] \label{theorem: CPY integral}
1) The $t$-module $E'$ is defined over $A$.

2) The point $\bv_\fs$ is an integral point in $E'(A)$.
\end{theorem}

The following examples were given in \cite[\S 6.1.2]{CPY19}. We refer the reader there for more examples.

\begin{example}
We consider $q=3$ and $\fs=(s_1=2,s_2=4)$. Then
\begin{align*}
E'_\theta=
\left(
\begin{array}{c c c c c c | c c c c }
\theta & 1 & 0 & 0 & 0 & 0 & 0 & 0 & 0 & 0  \\
0 & \theta & 1 & 0 & 0 & 0 & 0 & 0 & 0 & 0  \\
0 & 0 & \theta & 1 & 0 & 0 & 0 & 0 & 0 & 0  \\
0 & 0 & 0 & \theta & 1 & 0 & 0 & 0 & 0 & 0   \\
0 & 0 & 0 & 0 & \theta & 1 & 0 & 0 & 0 & 0  \\
\tau & 0 & 0 & 0 & 0 & \theta & -\tau & 0 & 0 & 0  \\
\hline
0 & 0 & 0 & 0 & 0 & 0 & \theta & 1 & 0 & 0  \\
0 & 0 & 0 & 0 & 0 & 0 & 0 & \theta & 1 & 0   \\
0 & 0 & 0 & 0 & 0 & 0 & 0 & 0 & \theta & 1   \\
0 & 0 & 0 & 0 & 0 & 0 & \tau & 0 & 0 & \theta  
\end{array}
\right)
\end{align*}
and 
	\[ \bv_\fs=(0,0,1,0,1,(\theta+2\theta^3),2,0,2,(2\theta+\theta^3))^\top. \]
\end{example}

\begin{example}
We consider $q=3$ and $\fs=(s_1=4,s_2=2)$. Then
\begin{align*}
E'_\theta=
\left(
\begin{array}{c c c c c c | c c }
\theta & 1 & 0 & 0 & 0 & 0 & 0 & 0   \\
0 & \theta & 1 & 0 & 0 & 0 & 0 & 0  \\
0 & 0 & \theta & 1 & 0 & 0 & \tau & 0  \\
0 & 0 & 0 & \theta & 1 & 0 & 0 & 0   \\
0 & 0 & 0 & 0 & \theta & 1 & \tau & 0  \\
\tau & 0 & 0 & 0 & 0 & \theta & (\theta+2\theta^3)\tau & 0 \\
\hline
0 & 0 & 0 & 0 & 0 & 0 & \theta & 1  \\
0 & 0 & 0 & 0 & 0 & 0 & \tau  & \theta 
\end{array}
\right)
\end{align*}
and 
	\[ \bv_\fs=(0,0,1,0,1,(\theta+2\theta^3),0,1)^\top. \]
\end{example}

For $1 \leq \ell < j$, we have defined the series
	\[ \mathfrak L(s_\ell,\dots,s_{j-1}):=\sum_{i_\ell > \dots > i_{j-1} \geq 0} (\Omega^{s_{j-1}} H_{s_{j-1}})^{(i_{j-1})} \dots (\Omega^{s_\ell} H_{s_\ell})^{(i_\ell)}, \]
	\[ \mathfrak L^*(s_\ell,\dots,s_{j-1}):=\sum_{i_\ell \geq \dots \geq i_{j-1} \geq 0} (\Omega^{s_{j-1}} H_{s_{j-1}})^{(i_{j-1})} \dots (\Omega^{s_\ell} H_{s_\ell})^{(i_\ell)}. \]
By \cite[5.5.3]{Cha14} we have
\begin{equation} \label{eq: MZV}
\left[\mathfrak L(s_\ell,\dots,s_{j-1}) \Omega^{-(s_\ell+\dots+s_{j-1})} \right](\theta)=\Gamma_{s_\ell} \dots \Gamma_{s_{j-1}} \zeta_A(s_\ell,\dots,s_{j-1}).
\end{equation}

We define the multiple zeta star values by
\begin{equation*} 
\zeta^*_A(s_1,\dots,s_{r}):=\sum \frac{1}{a_1^{s_1} \ldots a_r^{s_{r}}} \in K_\infty
\end{equation*}
where the sum runs through the set of tuples $(a_1,\ldots,a_r) \in A_+^r$ with $\deg a_1\geq\ldots\geq\deg a_r$.  Note that by \cite[Eq. (1)]{AT09} we have
\begin{equation*} 
\Gamma_{s_\ell} \dots \Gamma_{s_{j-1}} \zeta^*_A(s_\ell,\dots,s_{j-1})=\left[ \mathfrak L^*(s_\ell,\dots,s_{j-1}) \Omega^{-(s_\ell+\dots+s_{j-1})} \right](\theta).
\end{equation*}
We observe that these quantities can be completely determined by the relations 
\begin{align*} 
\zeta^*_A(s_r,\dots,s_1) = \sum_{\ell=2}^r (-1)^\ell \zeta_A(s_1,\dots,s_{\ell-1}) \zeta^*_A(s_r,\dots,s_\ell) + (-1)^{r-1} \zeta_A(s_1,\dots,s_r).
\end{align*}

We apply Theorem \ref{theorem: AT model} to this situation and obtain
\begin{theorem} \label{theorem: CPY model}
Recall that for $1 \leq \ell \leq r$, we put $d_\ell=s_\ell+\dots+s_r$. Then we have
\begin{align*}
\Log_{E'}^!(\bv_\fs)
=\delta_0 \begin{pmatrix}
(-1)^{r-1} \mathfrak L^*(s_r,\dots,s_1) \Omega^{-(s_1+\dots+s_r)} \\
(-1)^{r-2} \mathfrak L^*(s_r,\dots,s_2) \Omega^{-(s_2+\dots+s_r)} \\
\vdots \\
\mathfrak L^*(s_r) \Omega^{-s_r} 
\end{pmatrix}.
\end{align*}
In particular, for $1 \leq \ell \leq r$, the $(d_1+\dots+d_\ell)$th coordinate of $\Log_{E'}^!(\bv_\fs)$ equals $(-1)^{r-\ell} \Gamma_{s_\ell} \dots \Gamma_{s_r} \zeta^*_A(s_r\dots,s_\ell)$.
\end{theorem}


\subsection{Log-algebraic identities for $t$-modules connected to multiple polylogarithms at algebraic points} ${}$\par \label{sec: CM model}

For $\bu=(u_1,\dots,u_r) \in \overline K^r$ satisfying $|u_r|_{\infty} < q^{\frac{s_r q}{q-1}}$ and $|u_i|_{\infty} \leq q^{\frac{s_i q}{q-1}}$ for $1 \leq i \leq r-1$, Chang and Mishiba specialize the dual $t$-motives $\calM$ and $\calM'$ from the previous section to $(Q_1,\dots,Q_r) = (u_1,\dots,u_r)$ and thus get logarithmic interpretations for Carlitz star multiple polylogarithms (see \cite[Theorem 4.2.3]{CM19b} and also \cite[Theorem 3.3.7]{CGM19}).  They then use these polylogarithmic interpretations to get a logarithmic interpretation for MZV's; we will present a more direct way to recover MZV's using our techniques in \S \ref{S:Log Int for MZVs}.  In this section we show how our techniques recover Chang and Mishiba's result \cite[Theorem 4.2.3]{CM19b} (which only gives a certain coordinate of the logarithm) and that they also include the extra information given in Chang, Mishiba and the first author's result \cite[Theorem 3.3.7]{CGM19} (which gives all the coordinates of the logarithm).

We now define Carlitz (star) multiple polylogarithms, following as in \cite[\S 3.1]{CM19b}.  For any index $\fs=(s_{1},\dots,s_{r})\in \N^r$ we define the series
\begin{align*}
\text{Li}_\fs(z_1,\dots,z_r):=\sum_{i_1 > \dots > i_r \geq 0} \frac{z_1^{(i_1)} \dots z_r^{(i_r)}}{L_{i_1}^{s_1} \dots L_{i_r}^{s_r}}\in \power{\C_\infty}{z_{1},\dots,z_{r}}, \\
\text{Li}^*_\fs(z_1,\dots,z_r):=\sum_{i_1 \geq \dots \geq i_r \geq 0} \frac{z_1^{(i_1)} \dots z_r^{(i_r)}}{L_{i_1}^{s_1} \dots L_{i_r}^{s_r}}\in \power{\C_\infty}{z_{1},\dots,z_{r}},
\end{align*}
where $L_{0}:=1$ and $L_{i}:=(\theta-\theta^{q})\cdots (\theta-\theta^{q^{i}})$ for $i\in \N$.  The following formula is shown in \cite[Lemma 4.2.1]{CM19}:
\begin{align*}
& \text{Li}^*_{(s_r,\dots,s_1)}(z_r,\dots,z_1) \\
&= \sum_{\ell=2}^r (-1)^\ell \text{Li}_{(s_1,\dots,s_{\ell-1})}(z_1,\dots,z_{\ell-1}) \text{Li}^*_{(s_r,\dots,s_\ell)}(z_r,\dots,z_\ell) + (-1)^{r+1} \text{Li}_{(s_1,\dots,s_r)}(z_1,\dots,z_r).
\end{align*}
In particular, for $r=2$, we obtain
\begin{align*}
\text{Li}^*_{(s_2,s_1)}(z_2,z_1) =& \text{Li}_{s_1}(z_1) \text{Li}^*_{s_2}(z_2)- \text{Li}_{(s_1,s_2)}(z_1,z_2) \\
=& \text{Li}_{s_1}(z_1) \text{Li}_{s_2}(z_2)- \text{Li}_{(s_1,s_2)}(z_1,z_2).
\end{align*}
We then define $t$-deformed versions of $L_{i}$ as
\begin{equation}\label{D:LLdef}
\LL_{0}:=1\hbox{ and }\LL_{i}:=(t-\theta^{q})\cdots(t-\theta^{q^{i}}) \hbox{ for }i\in \N.
\end{equation}
We also define $t$-deformations of the $\Li$ and $\Li^*$ series as is done in \cite[\S 3.1]{CGM19} by setting
\begin{equation*} 
\fLi_\fs(t;z_1,\dots,z_r) := \sum_{i_1>\dots>i_r \geq 0} \frac{z_1^{q^{i_1}}\dots z_r^{q^{i_r}}}{\LL_{i_1}^{s_1}\dots \LL_{i_r}^{s_r}}\in \power{\C_\infty}{t,z_{1},\dots,z_{r}},
\end{equation*}
\begin{equation*} 
\fLis_\fs(t;z_1,\dots,z_r) := \sum_{i_1\geq \dots \geq i_r \geq 0} \frac{z_1^{q^{i_1}}\dots z_r^{q^{i_r}}}{\LL_{i_1}^{s_1}\dots \LL_{i_r}^{s_r}}\in \power{\C_\infty}{t,z_{1},\dots,z_{r}}.
\end{equation*}

Observe that if we set $\fQ = (Q_1,\dots, Q_r)$ from \S \ref{SS:AndThakModels} to be equal $\bu = (u_1,\dots, u_r)\in \overline K^n$, then we have the equalities
\[\fL(s_1,\dots,s_r) =  \Omega^{s_1 + \dots + s_r} \fLi(t;u_1,\dots,u_r),\]
\[\fL^*(s_1,\dots,s_r) =  \Omega^{s_1 + \dots + s_r} \fLis(t;u_1,\dots,u_r).\]
We set $\Phi_\bu$ equal to $\Phi$ from \S \ref{SS:AndThakModels} with $(Q_1,\dots,Q_r) = (u_1,\dots,u_r)$, and similarly for $\calM_\bu$ and $\calM'_\bu$.  Then, using the above equations we quickly deduce that the rigid analytic trivialization given by	
\begin{align*}
\Psi_\bu &=
\begin{pmatrix}
\Omega^{s_1+\dots+s_r} & 0 & 0 & \dots & 0 \\
\fLi_{s_1}(u_1) \Omega^{s_2+\dots+s_r} & \Omega^{s_2+\dots+s_r} & 0 & \dots & 0 \\
\vdots & \fLi_{s_2}(u_2) \Omega^{s_3+\dots+s_r} & \ddots & & \vdots \\
\vdots & & \ddots & \ddots & \vdots \\
\fLi_{(s_1,\dots,s_{r-1})}(u_1,\dots,u_{r-1}) \Omega^{s_r}  & \fLi_{(s_2,\dots,s_{r-1})}(u_2,\dots,u_{r-1}) \Omega^{s_r} & \dots & \Omega^{s_r}& 0 \\
\fLi_{(s_1,\dots,s_r)}(u_1,\dots,u_r)  & \fLi_{(s_2,\dots,s_r)}(u_2,\dots,u_r)  & \dots & \fLi_{s_r}(u_r)  & 1 
\end{pmatrix}.
\end{align*}
satisfies $\Psi_\bu \in \text{GL}_{r+1}(\bT)$ and 
	\[ \Psi_\bu^{(-1)}=\Phi_\bu \Psi_\bu. \]

The periods of $\calM_\bu$ are given by the matrix $\Upsilon_\bu=\Psi\inv_\bu$:
\begin{align*}
\Upsilon_\bu &=
\begin{pmatrix}
\Omega^{-(s_1+\dots+s_r)} & 0 & \dots & 0 \\
-\fLi^*_{s_1}(u_1) \Omega^{-(s_1+\dots+s_r)} & \Omega^{-(s_2+\dots+s_r)} & \dots & 0 \\
\vdots & & \ddots & \vdots \\
(-1)^r \fLi^*_{(s_r,\dots,s_1)}(u_r,\dots,u_1) \Omega^{-(s_1+\dots+s_r)} & \dots & -\fLi^*_{s_r}(u_r) \Omega^{-s_r} & 1 
\end{pmatrix}.
\end{align*}

Note that $\Upsilon_\bu \in \text{GL}_{r+1}(\bT)$ and 
\begin{align*}
\bff_\bu &=(0,\dots,0,u_r^{(-1)} (t-\theta)^{s_r}), \\
(\Psi_\bu)_{r+1} &=(\fLi_{(s_1,\dots,s_r)}(u_1,\dots,u_r),\fLi_{(s_2,\dots,s_r)}(u_2,\dots,u_r),\dots,\fLi_{s_r}(u_r)).
\end{align*}
In particular, they verify the hypothesis of Theorem \ref{theorem: main theorem} (a) by \cite[Lemma 5.3.1 and Theorem 5.5.2]{Cha14}. 

We can define the point $\bv_\bu \in E'_\bu(\overline K)$ as before. By the same calculations given in the proof of Theorem \ref{theorem: AT model} we see that it coincides with the point given in \cite[Equation (4.1.6)]{CM19b}  (see also \cite[Equation (3.3.1)]{CM19}):
\begin{align*}
\bv_\bu=\begin{pmatrix}
0 \\
\vdots \\
0 \\
(-1)^{r-1} u_r \dots u_1 \\
0 \\
\vdots \\
0 \\
(-1)^{r-2} u_r \dots u_2 \\
\vdots \\
0 \\
\vdots \\
0 \\
u_r 
\end{pmatrix}.
\end{align*}
Here for $1 \leq \ell \leq r$, the $(d_1+\dots+d_\ell)$th coordinate of $\bv_\bu$ equals $(-1)^{r-\ell} u_r \dots u_\ell$ and the other coordinates of $\bv_\bu$ vanish.  Applying Theorem \ref{theorem: main theorem} in this situation gives a refinement of \cite[Theorem 4.2.3]{CM19b},  (see also \cite{CM19}), and it recovers \cite[Theorem 3.3.7]{CGM19}.
\begin{theorem} \label{theorem: CM model}
Recall that for $1 \leq \ell \leq r$, we put $d_\ell=s_\ell+\dots+s_r$. Then we have
\begin{align*}
\Log_{E'_\bu}(\bv_\bu)
=\delta_0 \begin{pmatrix}
(-1)^{r-1} \fLi^*_{(s_r,\dots,s_1)}(u_r,\dots,u_1) \Omega^{-(s_1+\dots+s_r)} \\
(-1)^{r-2} \fLi^*_{(s_r,\dots,s_2)}(u_r,\dots,u_2) \Omega^{-(s_2+\dots+s_r)} \\
\vdots \\
\fLi^*_{s_r}(u_r) \Omega^{-s_r}
\end{pmatrix}.
\end{align*}
In particular, for $1 \leq \ell \leq r$, the $(d_1+\dots+d_\ell)$th coordinate of the $\Log_{E'_\bu}(\bv_\bu)$ equals $(-1)^{r-\ell} \Gamma_{s_\ell} \dots \Gamma_{s_r} \Li^*_{(s_r,\dots,s_\ell)}(u_r,\dots,u_\ell)$.
\end{theorem}

\begin{remark}
From the explicit formula for the point $\bv_\bu$ we see that it lies in the domain of convergence of  $\Log_{E'_\bu}$. Hence the split-logarithmic identity is indeed an actual logarithmic identity.
\end{remark}


\section{Star dual $t$-motives and application to MZV's} \label{S:Log Int for MZVs}

We see in \S \ref{S:Log for And Thak Models} that the Anderson-Thakur dual $t$-motive does not directly give a logarithmic interpretation for MZV's. In \cite{CGM19,CM19b} Chang, Green and Mishiba found a solution for this problem. Their method consisted of two steps. First, they find a logarithmic interpretation for Carlitz star multiple polylogarithms (see Theorem \ref{theorem: CM model}, also \cite{CM19}), then they form a linear combination of these polylogarithms which results in a MZV using the theory of fiber coproducts of $t$-motives (see \cite{CGM19,CM19b}). They raised the question whether one could find a more direct way to obtain a logarithmic interpretation for MZV's (see \cite[\S 1.4]{CM19b}).

In this section we give an affirmative answer to the above question of Chang and Mishiba and propose another logarithmic interpretation for MZV's which is much more direct. The key point is to introduce a new dual $t$-motive called the star dual $t$-motive so that MZV's are ``directly connected" to the associated $t$-module.

\subsection{Star dual $t$-motives and periods} ${}$\par

We always work with a tuple $\fs=(s_1,\dots,s_r) \in \mathbb N^r$ for $r \geq 1$. In what follows, we will specialize to $\fQ = (Q_1,\ldots,Q_r)=(H_{s_1},\ldots,H_{s_r})$ and keep the notation of \S \ref{sec: CPY}. 

\begin{remark}
We mention that all the results of this section still hold for any $\fQ=(Q_1,\ldots,Q_r) \in \overline K[t]^r$ satisfying the condition \eqref{eq: condition for Q}. The proofs can be adapted without modification.
\end{remark}

We set
\begin{align*}
\Phi^* &:=
\begin{pmatrix}
\Phi^*_{1,1} & &  \\
\vdots & \ddots & \\
\Phi^*_{r+1,1} & \dots & \Phi^*_{r+1,r+1}
\end{pmatrix} \in \Mat_{r+1}(\overline K[t])
\end{align*}
where for $1 \leq \ell \leq j \leq r+1$, 
\begin{equation} \label{eq: Phi star}
\Phi^*_{j,\ell}=(-1)^{j-\ell} \prod_{\ell \leq k <j} Q_k^{(-1)} (t-\theta)^{s_\ell+\dots+s_r}.
\end{equation}
We also set
\begin{align*}
\Phi'^* &:=
\begin{pmatrix}
\Phi^*_{1,1} & &  \\
\vdots & \ddots & \\
\Phi^*_{r,1} & \dots & \Phi^*_{r,r}
\end{pmatrix} \in \Mat_r(\overline K[t]).
\end{align*}
Let $\calM^*$ and $\calM'^*$ be the dual $t$-motives defined by $\Phi^*$ and $\Phi'^*$ respectively.  We define
\begin{align*}
\Psi^* &:=
\begin{pmatrix}
\Omega^{(s_1+\dots+s_r)} & 0 & 0 & \dots & 0 \\
-\mathfrak L^*(s_1) \Omega^{(s_2+\dots+s_r)} & \Omega^{(s_2+\dots+s_r)} & 0 & \dots & 0 \\
\vdots & -\mathfrak L^*(s_2) \Omega^{(s_3+\dots+s_r)} & \ddots & & \vdots \\
\vdots & & \ddots & \ddots & \vdots \\
(-1)^{r-1} \mathfrak L^*(s_1,\dots,s_{r-1}) \Omega^{s_r}  & (-1)^{r-2} \mathfrak L^*(s_2,\dots,s_{r-1}) \Omega^{s_r} & \dots & \Omega^{s_r}& 0 \\
(-1)^r \mathfrak L^*(s_1,\dots,s_r)  & (-1)^{r-1} \mathfrak L^*(s_2,\dots,s_r) & \dots & -\mathfrak L^*(s_r) & 1 
\end{pmatrix}.
\end{align*}
Then if we set $\Upsilon^*=(\Psi^*)^{-1}$, then we use Lemma \ref{L:Lstarinv} to get
\begin{align*}
\Upsilon^* &=
\begin{pmatrix}
\Omega^{-(s_1+\dots+s_r)} & 0 & 0 & \dots & 0 \\
\mathfrak L(s_1) \Omega^{-(s_1+\dots+s_r)} & \Omega^{-(s_2+\dots+s_r)} & 0 & \dots & 0 \\
\vdots & \mathfrak L(s_2) \Omega^{-(s_2+\dots+s_r)} & \ddots & & \vdots \\
\vdots & & \ddots & \ddots & \vdots \\
\mathfrak L(s_{r-1},\dots,s_1) \Omega^{-(s_1+\dots+s_r)}  & \mathfrak L(s_{r-1},\dots,s_2) \Omega^{-(s_2+\dots+s_r)} & \dots & \Omega^{-s_r} & 0 \\
\mathfrak L(s_r,\dots,s_1) \Omega^{-(s_1+\dots+s_r)} & \mathfrak L(s_r,\dots,s_2) \Omega^{-(s_2+\dots+s_r)} & \dots & \mathfrak L(s_r) \Omega^{-s_r} & 1 
\end{pmatrix}.
\end{align*}
Note that $\Psi^*$ and $\Upsilon^*$ belongs to $\text{GL}_{r+1}(\bT)$. Further, by Lemma \ref{lemma: equality star} we obtain
	\[ \Psi^{* (-1)}=\Phi^* \Psi^*. \]
	
\begin{lemma}\label{L:MZV and L series}
The value at $t=\theta$ of the last line of $\Upsilon^*$ is 
	\[ (\Gamma_{s_1} \dots \Gamma_{s_r} \zeta_A(s_r,\dots,s_1),\Gamma_{s_2} \dots \Gamma_{s_r} \zeta_A(s_r,\dots,s_2),\dots,\Gamma_{s_r} \zeta_A(s_r),1). \]
\end{lemma}

\begin{proof}
This follows immediately from Equality \eqref{eq: MZV}.
\end{proof}

Let $\bm=\{m_1,\dots,m_r\}$ be the $\overline K[t]$-basis of $\calM'^*$ with respect to the action of $\sigma$ represented by $\Phi'^*$. It is not hard to check that $\calM'^*$ is a free left $\overline K[\sigma]$-module of rank $d=(s_1+\dots+s_r)+(s_2+\dots+s_r)+ \dots + s_r$ and
	\[ \bw=\{w_1,\ldots,w_d\}:=\{(t-\theta)^{s_1+\dots+s_r-1} m_1,\dots, m_1,\dots,(t-\theta)^{s_r-1} m_r,\dots,m_r \} \]
is a $\overline K[\sigma]$-basis of $\calM'^*$. We further observe that $(t-\theta)^\ell \calM'^*/\sigma \calM'^* =(0)$ for $\ell \gg 0$. 

We denote by $E'^*$ the $t$-module defined by the dual $t$-motive $\calM'^*$ given by the matrix $\Phi'^*$. We can write down explicitly the maps $\delta_0:\calM'^* \to \Mat_{d \times 1}(\overline K)$ and $\delta_1:\calM'^* \to \Mat_{d \times 1}(\overline K)$. For the convenience of the reader we present the former map which is the same as that for the Anderson-Thakur dual $t$-motives. Let $m \in \calM'^*=\overline K[t]m_1+\dots+\overline K[t]m_r$. Then we can write
\begin{equation}\label{E:m in sigma basis}
m=\sum_{\ell=1}^r (c_{d_\ell-1,\ell} (t-\theta)^{d_\ell-1}+\dots+c_{0,\ell}+F_\ell(t) (t-\theta)^{d_\ell}) m_\ell,
\end{equation}
with $c_{i,\ell} \in \overline K$ and $F_\ell(t) \in \overline K[t]$. Then
\begin{equation}\label{D:delta0starmodel}
\delta_0(m):=(c_{d_1-1,1},\dots,c_{0,1},\dots,c_{d_r-1,r},\dots,c_{0,r})^\top.
\end{equation}

\subsection{Integrality properties} ${}$\par
Next, we consider 
	\[ \alpha(\calM^*) = \Phi^*_{r+1,1} m_1 + \dots + \Phi^*_{r+1,r} m_r \in \calM'^*/(\sigma-1)\calM'^* \] which corresponds to a certain point $\bv^*_\fs:= \delta_1(\alpha(\calM^*)) \in E'^*(\overline K)$. 
	
In this section we prove integrality properties of the Anderson $t$-module $E'^*$ and the point $\bv^*_\fs \in E'(\overline K)$ which will be used later to deduce a logarithmic interpretation for $\nu$-adic MZV's from that for MZV's (see Theorem \ref{theorem: nu adic}). Our result is inspired by \cite[Theorem 5.3.4]{CPY19} (see Theorem \ref{theorem: CPY integral}). Indeed, its proof can be adapted without much modification. For the convenience of the reader we write it down completely below.

\begin{proposition} \label{prop: integral}
Recall that $\bw=\{w_1,\dots,w_d\}$ denotes the $\overline K[\sigma]$-basis
	\[ \{(t-\theta)^{s_1+\dots+s_r-1} m_1,\dots,(t-\theta) m_1, m_1,\dots,(t-\theta)^{s_r-1} m_r,\dots,(t-\theta) m_r,m_r \} \]
of $\calM'^*$. Let $\Xi$ be the set of all the elements of $\calM'^*$ of the form $\sum_{i=1}^d h_i w_i$ where $h_i=\sum_n \sigma^n u_{n,j}$ with $u_{n,j} \in A$. 

Then for $g \in A[t]$ and $1 \leq \ell \leq r$, we have $g m_\ell \in \Xi$. 
\end{proposition}

\begin{proof}
Recall that we have put $d_\ell=s_\ell+\dots+s_r$ for $1 \leq \ell \leq r$. We claim that there exist polynomials $g_{\ell,1},\dots,g_{\ell,\ell-1} \in A[t]$ such that 
	\[ (t-\theta)^{d_\ell} m_\ell=\sigma (g_{\ell,1} m_1+\dots+g_{\ell,\ell-1} m_{\ell-1}+m_\ell). \]
The proof is by induction on $\ell$. For $\ell=1$, we have $(t-\theta)^{d_1} m_1=\sigma m_1$, and the claim is clear. Suppose that we have proved the claim for $1 \leq i < \ell$. We now show that the claim is true for $\ell$. In fact, since $\sigma m_\ell=Q_{\ell,1}^{* (-1)} (t-\theta)^{d_1} m_1+\dots+Q_{\ell,\ell-1}^{*(-1)} (t-\theta)^{d_{\ell-1}} m_{\ell-1}+(t-\theta)^{d_\ell} m_\ell$ for explicit polynomials $Q_{\ell,j}^* \in A[t]$ given in \eqref{eq: Phi star}, we get
	\[ (t-\theta)^{d_\ell} m_\ell=\sigma m_\ell-Q_{\ell,1}^{* (-1)} (t-\theta)^{d_1} m_1-\dots-Q_{\ell,\ell-1}^{*(-1)} (t-\theta)^{d_{\ell-1}} m_{\ell-1}. \]
By induction it follows that
\begin{align*}
(t-\theta)^{d_\ell} m_\ell &=\sigma m_\ell-Q_{\ell,1}^{* (-1)} (t-\theta)^{d_1} m_1-\dots-Q_{\ell,\ell-1}^{*(-1)} (t-\theta)^{d_{\ell-1}} m_{\ell-1} \\
&=\sigma m_\ell - \sum_{i=1}^{\ell-1} Q_{\ell,i}^{* (-1)} \sigma (g_{i,1} m_1+\dots+g_{i,i-1} m_{i-1}+m_i) \\
&=\sigma m_\ell - \sum_{i=1}^{\ell-1} \sigma Q_{\ell,i}^* (g_{i,1} m_1+\dots+g_{i,i-1} m_{i-1}+m_i) \\
&=\sigma (m_\ell - \sum_{i=1}^{\ell-1} Q_{\ell,i}^* (g_{i,1} m_1+\dots+g_{i,i-1} m_{i-1}+m_i)).
\end{align*}
The proof of the claim is now complete.

We are now ready to show by induction on $\ell$ that for $g \in A[t]$, we have $g m_\ell \in \Xi$. We first assume that $\ell=1$. We show by induction on the degree of $g$ that $g m_1 \in \Xi$. It is clear that if $\deg g=0$, then the claim is true. Let $g \in A[t]$ with $\deg g>0$. We divide $g$ by $(t-\theta)^{d_1}$ and write 
	\[ g=h(t-\theta)^{d_1}+\sum_{j=0}^{d_1-1} a_j (t-\theta)^j \] 
with $h \in A[t]$ and $a_0,\dots,a_{d_1-1} \in A$. Since $\sigma m_1=(t-\theta)^{d_1} m_1$, it follows that
\begin{align*}
g m_1 &=h(t-\theta)^{d_1} m_1+\sum_{j=0}^{d_1-1} a_j (t-\theta)^j m_1 \\
&= \sigma h^{(1)} m_1 + \sum_{j=0}^{d_1-1} a_j (t-\theta)^j m_1.
\end{align*} 
Since $\deg h<\deg g$, by induction, $h^{(1)} m_1 \in \Xi$. Since $a_0,\dots,a_{d_1-1} \in A$, the sum $\sum_{j=0}^{d_1-1} a_j (t-\theta)^j m_1$ belongs to $\Xi$. Hence we conclude that $g m_1 \in \Xi$.

Now we consider $1< \ell \leq r$ and suppose that $g m_i \in \Xi$ for $1 \leq i < \ell$. We show by induction that $g m_\ell \in \Xi$. We divide $g$ by $(t-\theta)^{d_\ell}$ and write 
	\[ g=h(t-\theta)^{d_\ell}+r, \quad \text{with } h,r \in A[t] \text{ and } \deg r<d_\ell. \] 
We have seen that there exist polynomials $g_{\ell,1},\dots,g_{\ell,\ell-1} \in A[t]$ such that 
	\[ (t-\theta)^{d_\ell} m_\ell=\sigma (g_{\ell,1} m_1+\dots+g_{\ell,\ell-1} m_{\ell-1}+m_\ell). \]
It follows that
\begin{align*}
g m_\ell &=h(t-\theta)^{d_\ell} m_\ell+r m_\ell \\
&= \sigma h^{(1)} (g_{\ell,1} m_1+\dots+g_{\ell,\ell-1} m_{\ell-1}+m_\ell) +r m_\ell \\
&= \sigma h^{(1)} m_\ell + \sigma h^{(1)} (g_{\ell,1} m_1+\dots+g_{\ell,\ell-1} m_{\ell-1}) + r m_\ell.
\end{align*} 
The first and second terms belong to $\Xi$ by induction. Since $r \in A[t]$ and $\deg r<d_\ell$, the last term also belongs to $\Xi$. We conclude that $g m_\ell \in \Xi$ and the proof is finished.
\end{proof}

We prove an analogue version of \cite[Theorem 5.3.4]{CPY19} (see Theorem \ref{theorem: CPY integral}):
\begin{proposition}
1) The $t$-module $E'^*$ is defined over $A$.

2) The point $\bv^*_\fs$ is an integral point in $E'^*(A)$.
\end{proposition}

\begin{proof}
1) We keep the notation of Proposition \ref{prop: integral}. By Proposition \ref{prop: integral} we see that for $1 \leq i \leq d$, we have $t w_i \in \Xi$ which means $t w_i=\sum_{i=1}^d h_i w_i$ for some $h_i=\sum_n \sigma^n u_{n,j}$ with $u_{n,j} \in A$. Thus $E'^*$ is defined over $A$.

2) We remark that $\delta_1(\Xi) \in E'^*(A)$. Since $\bv^*_\fs=\delta_1(\Phi^*_{r+1,1} m_1 + \dots + \Phi^*_{r+1,r} m_r)$, it is sufficient to see that all the termes $\Phi^*_{r+1,1} m_1,\dots,\Phi^*_{r+1,r} m_r$ belongs to $\Xi$. In fact, let $1 \leq \ell \leq r$, by the proof of Proposition \ref{prop: integral} there exist polynomials $g_{\ell,1},\dots,g_{\ell,\ell-1} \in A[t]$ such that 
	\[ (t-\theta)^{d_\ell} m_\ell=\sigma (g_{\ell,1} m_1+\dots+g_{\ell,\ell-1} m_{\ell-1}+m_\ell). \]
Recall that $\Phi^*_{r+1,\ell}=Q^{*(-1)}_{r+1,\ell} (t-\theta)^{d_\ell}$ for an explicit polynomial $Q^*_{r+1,\ell} \in A[t]$ given in \eqref{eq: Phi star}. This implies that
\begin{align*}
\Phi^*_{r+1,\ell} m_\ell &=Q^{*(-1)}_{r+1,\ell} (t-\theta)^{d_\ell} m_\ell \\
&= \sigma Q^*_{r+1,\ell}(g_{\ell,1} m_1+\dots+g_{\ell,\ell-1} m_{\ell-1}+m_\ell).
\end{align*}
We conclude that $\Phi^*_{r+1,\ell} m_\ell \in \Xi$. The proof is complete.
\end{proof}

\subsection{Logarithm coefficients} ${}$\par

The coefficients of the logarithm series can be calculated following \cite{ANDTR20a}. We set
\begin{align*}
\Theta^*:=((\Phi^*)^{-1})^\top &=
\begin{pmatrix}
\Theta^*_{1,1} & \dots & \Theta^*_{1,r+1}  \\
& \ddots & \vdots \\
& & \Theta^*_{r+1,r+1}
\end{pmatrix} \in \Mat_{r+1}(\overline K(t))
\end{align*}
where 
	\[\Theta^*_{i,i}=\frac{1}{(t-\theta)^{s_i+\dots+s_r}} \] 
and for $1 \leq i < r+1$,
	\[ \Theta^*_{i,i+1}=\frac{Q_i^{(-1)}}{(t-\theta)^{s_{i+1}+\dots+s_r}}. \]
The other coefficients $\Theta^*_{i,j}$ vanish. 

We set
\begin{align*}
\Theta'^*:=((\Phi'^*)^{-1})^\top &=
\begin{pmatrix}
\Theta^*_{1,1} & \dots & \Theta^*_{1,r}  \\
& \ddots & \vdots \\
& & \Theta^*_{r,r}
\end{pmatrix} \in \Mat_r(\overline K(t)).
\end{align*}
If we write 
	\[ \Log_{E'^*}=\sum_{n \geq 0} P^*_n \tau^n, \]
then by \cite[Proposition 2.2]{ANDTR20a}, for $n \geq 0$, the $n$th coefficient of the logarithm series of $E^*$ evaluated at $\bv \in \overline K^d$ is given by
	\[ P^*_n \bv\twistk{n} = \delta_0(\Theta'^{*(1)} \dots \Theta'^{*(n)}\iota\inv (\bv)\twistk{n}). \]

\subsection{Logarithmic interpretations for MZV's} ${}$\par \label{sec:MZV}

Note that 
\begin{align*}
\bff^* &=((-1)^rQ_1^{(-1)} \dots Q_r^{(-1)} (t-\theta)^{s_1+\dots+s_r},\dots,-Q_r^{(-1)} (t-\theta)^{s_r}), \\
\Psi^*_{r+1} &=((-1)^r \mathfrak L^*(s_1,\dots,s_r),(-1)^{r-1} \mathfrak L^*(s_2,\dots,s_r),\dots,-\mathfrak L^*(s_r)).
\end{align*}
In particular, they verify the hypothesis of Theorem \ref{theorem: main theorem} by \cite[Lemma 5.3.1]{Cha14}. 

Theorem \ref{theorem: main theorem} implies the following theorem:
\begin{theorem} \label{theorem: star model}
Recall that for $1 \leq \ell \leq r$, we put $d_\ell:=s_\ell+\dots+s_r$. Then we have
\begin{align*}
\Log_{E'^*}^!(\bv^*_\fs)
=\delta_0 \begin{pmatrix}
- \mathfrak L(s_r,\dots,s_1) \Omega^{-(s_1+\dots+s_r)} \\
- \mathfrak L(s_r,\dots,s_2) \Omega^{-(s_2+\dots+s_r)} \\
\vdots \\
- \mathfrak L(s_r) \Omega^{-s_r} 
\end{pmatrix}.
\end{align*}

In particular, for $1 \leq \ell \leq r$, the $(d_1+\dots+d_\ell)$th coordinate of the $\Log_{E'^*}^!(\bv^*_\fs)$ equals $- \Gamma_{s_\ell} \dots \Gamma_{s_r} \zeta_A(s_r\dots,s_\ell)$.
\end{theorem}

\begin{proof}
We first estimate the domain of convergence of $\Log_{E'^*}$.  We observe that the (lower triangular) matrix $(\Phi'^*)^{-1}$ above agrees with $\Phi'^{-1}$ from \cite[\S 4.2]{CGM19} along the main diagonal, and that the first subdiagonal agrees up to a factor of $(-1)$, while $(\Phi'^*)^{-1}_{i,j} = 0$ for all the other (below the subdiagonal) entries.  This allows us to use the degree estimates given in \cite[Proposition 4.1.3]{CGM19} for the matrix $\Phi'^{-1}$ for our logarithm series $\Log_{E'^*}$ for the matrix $(\Phi'^*)^{-1}$.  Indeed, following the notation in the proof of \cite[Proposition 4.1.3]{CGM19} we fix $\bw = (w_1, \dots, w_r)\in \Mat_{1\times r}(\C_\infty[t])$ with
\[w_i = y_{i,1}(t-\theta)^{d_{i}-1}+y_{i,2}(t-\theta)^{d_{i}-2} +\cdots+y_{i,d_{i}}, \quad y_{i,j}\in \C_\infty, \quad 1\leq i\leq r.\]
Then, we see that the degree estimates for the entries of $\bw \twistk{n} \prod_{1 \leq k \leq n} (\Phi'^{-1}) \twistk{n+1-k}$ coincide with estimates for our $\bw \twistk{n} \prod_{1 \leq k \leq n} ((\Phi'^*)^{-1}) \twistk{n+1-k}$ for each term which involves only diagonal or subdiagonal entries of $(\Phi'^*)^{-1}$.  On the other hand, each term which involves any other entry of $(\Phi'^*)^{-1}$ will be identically zero, since the sub-sub-diagonal coordinates of $(\Phi'^*)^{-1}$ are all zero.  Thus, the formula for the degree of the $\ell$th component of $\bw \twistk{n} \prod_{1 \leq k \leq n} ((\Phi'^*)^{-1}) \twistk{n+1-k}$ will be a subsum of the formula for the degree of the $\ell$th component of $\bw \twistk{n} \prod_{1 \leq k \leq n} (\Phi'^{-1}) \twistk{n+1-k}$.  In particular, the degree estimates for $\bw \twistk{n} \prod_{1 \leq k \leq n} (\Phi'^{-1}) \twistk{n+1-k}$ will also hold for $\bw \twistk{n} \prod_{1 \leq k \leq n} ((\Phi'^*)^{-1}) \twistk{n+1-k}$, since they are bounded above by the maximum of these terms in this sum (see \cite[Proposition 4.1.3]{CGM19} for more details).  This allows us to conclude using \cite[Lemma 4.2.1]{CGM19} that $\Log_{E'^*}(\by)$ converges as long as $\lVert Q_i \rVert_{1} \leq q^{\frac{s_i q}{q-1}}$ and $\by = (y_{1,1},\dots,y_{1,d_1},\dots,y_{r,1},\dots,y_{r,d_r})^\top$ satisfies the condition that $|y_{i,j}|_{\infty} < q^{j+\frac{d_i}{q-1}}$ for each $1 \leq i \leq r$ and $1 \leq j \leq d_i$.  To summarize, the radius of convergence of $\Log_{E'^*}$ is at least as large as that of $\Log_{E'}$.

Next, we turn to analyzing
\[\alpha(\calM) = (-1)^rQ_1^{(-1)} \cdots Q_r^{(-1)} (t-\theta)^{s_1+\dots+s_r}m_1 + \cdots + Q_{r-1}^{(-1)} Q_r^{(-1)} (t-\theta)^{s_{r-1}+s_r}m_{r-1} - Q_r^{(-1)} (t-\theta)^{s_r}m_r.\]
From the defining equation for $\Phi'^*$ we see that 
\[\sigma m_r =  (-1)^{r-1}Q_1^{(-1)} \cdots Q_{r-1}^{(-1)} (t-\theta)^{s_1+\dots+s_r}m_1 + \cdots - Q_{r-1}^{(-1)} (t-\theta)^{s_{r-1}+s_r}m_{r-1} + (t-\theta)^{s_r}m_r.\]
From this we conclude that 
\begin{equation} \label{eq: w_s}
\alpha(\calM) = -Q_r\twistinv \sigma(m_r)=-\sigma(Q_rm_r).
\end{equation}
From \cite[(3.7.3)]{AT90} we know that $\lVert Q_{i}\rVert = \lVert H_{s_i}\rVert < |\theta|_\infty ^{\tfrac{s_iq}{q-1}},$ so the conditions of Theorem \ref{theorem: main theorem} (b) are satisfied, which proves the first statement of the theorem.

The second statement of the theorem follows immediately from Lemma \ref{L:MZV and L series} and the definition of $\delta_0$.  However, for the convenience of the reader, we write down a direct proof of the second part of Theorem \ref{theorem: star model}.  We wish to prove that for $1 \leq \ell \leq r$, the $(d_1+\dots+d_\ell)$th coordinate of the $\Log_{E'^*}^!(\bv^*_\fs)$ denoted by $\nu_\ell$ equals to $- \Gamma_{s_\ell} \dots \Gamma_{s_r} \zeta_A(s_r\dots,s_\ell)$.

Recall that for $i \geq 0$, the $i$th coefficient of the logarithm series of $E'^*$ evaluated at $\bv \in \overline K^d$ is given by
	\[ P_i^* \bv\twistk{i} = \delta_0(\Theta'^{*(1)} \dots \Theta'^{*(i)}\iota\inv (\bv)\twistk{i}). \]
Let us denote the matrix $B_i = \Theta'^{*(1)} \dots \Theta'^{*(i)}$.  Then we quickly see that, $B_i[\ell j]=0$ if $\ell>j$ ($B[\ell j]$ denotes the $(\ell,j)$th entry of a matrix $B$). Further, if $\ell=j$, then
	\[ 	B_i[\ell j]=\frac{1}{\LL_i^{s_j+\dots+s_r}}, \]
where we recall the definition of $\LL_i$ from \eqref{D:LLdef}.  For $1 \leq \ell < j \leq r$, we get
	\[ B_i[\ell j]= \sum_{0 \leq i_\ell < \dots < i_{j-1}<i} \frac{Q_\ell^{(i_\ell)} \dots Q_{j-1}^{(i_{j-1})}}{\LL_{i_\ell}^{s_\ell} \dots \LL^{s_{j-1}}_{i_{j-1}} \LL_i^{s_j+\dots+s_r}}. \]
	
We wish to find the $(d_1 + \dots + d_\ell)$th coordinate of $\sum_{i \geq 0} P_i^* {\bv^*_\fs}^{(i)}$.  Recall that
\[\sum_{i \geq 0} P_i^* {\bv^*_\fs}^{(i)} = \sum_{i \geq 0}\delta_0(\Theta'^{*(1)} \dots \Theta'^{*(i)}\iota\inv (\bv^*_\fs)\twistk{i}),\]

We calculate that the $\ell$th coordinate of $B_i\iota\inv (\bv^*_\fs)\twistk{i} = B_i(f_1,\dots,f_r)^\top$ equals (here we understand $Q_\ell^{(i_\ell)} \dots Q_{j-1}^{(i_{j-1})}=1$ if $\ell=j$)
\begin{align*}
&\sum_{i \geq 0} \sum_{j=\ell}^r \sum_{0 \leq i_\ell < \dots < i_{j-1}<i} \frac{Q_\ell^{(i_\ell)} \dots Q_{j-1}^{(i_{j-1})}}{\LL_{i_\ell}^{s_\ell} \dots \LL^{s_{j-1}}_{i_{j-1}} \LL_i^{s_j+\dots+s_r}}  \left( (-1)^{r+1-j} \prod_{j \leq k <r+1} Q_k^{(-1)} (t-\theta)^{s_j+\dots+s_r} \right)^{(i)} \\
=&(-1)^{r+1-j} \prod_{j \leq k <r+1} Q_k^{(-1)} (t-\theta)^{s_j+\dots+s_r} \\
&+\sum_{i \geq 1} \sum_{j=\ell}^r \sum_{0 \leq i_\ell < \dots < i_{j-1}<i} \frac{Q_\ell^{(i_\ell)} \dots Q_{j-1}^{(i_{j-1})}}{\LL_{i_\ell}^{s_\ell} \dots \LL^{s_{j-1}}_{i_{j-1}} \LL_i^{s_j+\dots+s_r}}  \left( (-1)^{r+1-j} \prod_{j \leq k <r+1} Q_k^{(-1)} (t-\theta)^{s_j+\dots+s_r} \right)^{(i)}.
\end{align*}
As $\delta_0\left ((-1)^{r+1-j} \prod_{j \leq k <r+1} Q_k^{(-1)} (t-\theta)^{s_j+\dots+s_r}\right )=0$, we omit this term from our calculation and continue with
\begin{align*}
& \sum_{i \geq 1} \sum_{j=\ell}^r \sum_{0 \leq i_\ell < \dots < i_{j-1}<i} \frac{Q_\ell^{(i_\ell)} \dots Q_{j-1}^{(i_{j-1})}}{\LL_{i_\ell}^{s_\ell} \dots \LL^{s_{j-1}}_{i_{j-1}} \LL_i^{s_j+\dots+s_r}}  \left( (-1)^{r+1-j} \prod_{j \leq k <r+1} Q_k^{(-1)} (t-\theta)^{s_j+\dots+s_r} \right)^{(i)} \\
&=\sum_{i \geq 1} \sum_{j=\ell}^r \sum_{0 \leq i_\ell < \dots < i_{j-1} \leq i-1} (-1)^{r+1-j} \frac{Q_\ell^{(i_\ell)} \dots Q_{j-1}^{(i_{j-1})} \cdot Q_j^{(i-1)} \dots Q_r^{(i-1)}}{\LL_{i_\ell}^{s_\ell} \dots \LL^{s_{j-1}}_{i_{j-1}} \cdot \LL_{i-1}^{s_j+\dots+s_r}} \\
&=\sum_{i \geq 0} \sum_{j=\ell}^r \sum_{0 \leq i_\ell < \dots < i_{j-1} \leq i} (-1)^{r+1-j} \frac{Q_\ell^{(i_\ell)} \dots Q_{j-1}^{(i_{j-1})} \cdot Q_j^{(i)} \dots Q_r^{(i)}}{\LL_{i_\ell}^{s_\ell} \dots \LL^{s_{j-1}}_{i_{j-1}} \cdot \LL_i^{s_j+\dots+s_r}} \\
&=-\sum_{i \geq 0} \sum_{0 \leq i_\ell < \dots < i_{r-1} < i} \frac{Q_\ell^{(i_\ell)} \dots Q_{r-1}^{(i_{r-1})} Q_r^{(i)}}{\LL_{i_\ell}^{s_\ell} \dots \LL^{s_{r-1}}_{i_{r-1}} \LL_i^{s_r}}.
\end{align*}
Finally, we observe by \eqref{D:delta0starmodel} that finding the $(d_1 + \dots + d_\ell)$th coordinate of a vector $\delta_0((g_1,\dots,g_r)^\top)$ for $1\leq \ell\leq r$ and $g_i\in \TT$ boils down to evaluating $g_\ell$ at $t=\theta$.  Thus 
	\[ -\sum_{i \geq 0} \sum_{0 \leq i_\ell < \dots < i_{r-1} < i} \frac{Q_\ell^{(i_\ell)} \dots Q_{r-1}^{(i_{r-1})} Q_r^{(i)}}{\LL_{i_\ell}^{s_\ell} \dots \LL^{s_{r-1}}_{i_{r-1}} \LL_i^{s_r}}\bigg|_{t=\theta}=-\Gamma_{s_r} \dots \Gamma_{s_\ell} \zeta_A(s_r,\dots,s_\ell), \]
by \eqref{eq: MZV} and this finishes the proof.
\end{proof}

\begin{remark}
Jing Yu \cite{CM20} suggested that the logarithmic interpretation for MZV's obtained in Theorem \ref{theorem: star model} could be viewed as a  ``nice" integral interpretation for the MZV's, thus it may be called a linear form of Anderson logarithms.
\end{remark}

\subsection{Logarithmic interpretations for $\nu$-adic MZV's} ${}$\par \label{sec: v-adic}

Throughout this section we fix a finite place $\nu$ of $K$ which corresponds to an irreducible monic polynomial still denoted by $\nu$ of $A$. We let $K_\nu$ be the completion of $K$ at $\nu$ and let $\C_\nu$ be the completion of an algebraic closure of $K_\nu$. Let $\lvert\cdot\rvert_\nu$ be the normalized $\nu$-adic absolute value on $\C_\nu$. This $\nu$-adic absolute value extends naturally to matrices with entries in $\C_\nu$. 

This section aims to present a logarithmic interpretation for $\nu$-adic MZV's. For the depth one case, i.e. for $\nu$-adic zeta values, this was done by Anderson and Thakur (see \cite[Theorem 3.8.3]{AT90}). We mention that Chang and Mishiba in \cite{CM19} gave another interpretation for these values by combining the Anderson-Thakur dual $t$-motives and the notion of fiber coproducts. We show that their arguments can carry over to our setting. Consequently, we deduce from Theorem \ref{theorem: star model} a logarithmic interpretation of $\nu$-adic MZV's (see \cite[Theorem 6.2.4]{CM19b}). 

In what follows, we always work with a tuple $\fs=(s_1,\dots,s_r) \in \mathbb N^r$ for $r \geq 1$. We work with the $t$-module $E'^*$ introduced in \S \ref{S:Log Int for MZVs} and keep the notation of this section.

\begin{proposition} \label{proposition: nu adic convergence}
For any $\bv \in E'^*(\C_\nu)$ with $\lvert \bv \rvert_\nu<1$, $\Log_{E'^*}(\bv)$ converges $\nu$-adically in $\Lie_{E'^*}(\C_\nu)$.
\end{proposition}
	
\begin{proof}
We write
	\[ \Log_{E'^*}(\bv)=\sum_{i \geq 0} P_i^* \bv\twistk{i} \]
Recall that for $i \geq 0$, the $i$th coefficient of the logarithm series of $E'^*$ evaluated at $\bv \in \overline K^d$ is given by
	\[ P_i^* \bv\twistk{i} = \delta_0(\Theta'^{*(1)} \dots \Theta'^{*(i)}\iota\inv (\bv)\twistk{i}). \]
Here the matrix $B_i = \Theta'^{*(1)} \dots \Theta'^{*(i)}$ is given as follows. We have $B_i[\ell j]=0$ if $\ell>j$ ($B[\ell j]$ denotes the $(\ell,j)$th entry of a matrix $B$). Further, if $\ell=j$, then
	\[ 	B_i[\ell j]=\frac{1}{\LL_i^{s_j+\dots+s_r}}, \]
where we recall the definition of $\LL_i$ from \eqref{D:LLdef}.  For $1 \leq \ell < j \leq r$, we get
	\[ B_i[\ell j]= \sum_{0 \leq i_\ell < \dots < i_{j-1}<i} \frac{Q_\ell^{(i_\ell)} \dots Q_{j-1}^{(i_{j-1})}}{\LL_{i_\ell}^{s_\ell} \dots \LL^{s_{j-1}}_{i_{j-1}} \LL_i^{s_j+\dots+s_r}}. \]

We consider $w_k=(t-\theta)^s m_j$ (with $1 \leq j \leq r$, $0 \leq s <d_j$) which is an element of the $\overline K[\sigma]$-basis 
	\[ \bw =\{(t-\theta)^{s_1+\dots+s_r-1} m_1,\dots,(t-\theta) m_1, m_1,\dots,(t-\theta)^{s_r-1} m_r,\dots,(t-\theta) m_r,m_r \} \]
of $\calM'^*$. We note that $k=d_1+\dots+d_{j-1}+s$. The vector $w_k$ corresponds to the $k$th vector in the canonical basis of $\overline K^d$. 

Letting $P_i^*[k',k]$ the $(k',k)$th entry of $P_i^*$, we get
\begin{align*}
(P_i^*[1,k],\dots,P_i^*[d_1+\dots+d_r,k])^\top &=P_i^* w_k \\
&= \delta_0(\Theta'^{*(1)} \dots \Theta'^{*(i)} \iota\inv (w_k)\twistk{i}) \\
&= \delta_0((B_i[1,j](t-\theta^{q^i})^s,\dots,B_i[r,j](t-\theta^{q^i})^s)^\top).
\end{align*} 
Recall that the map $\delta_0$ is given explicitly by \eqref{D:delta0starmodel}.  Since we may rewrite the first $d_\ell$ terms in each coordinate of equation \eqref{E:m in sigma basis} in terms of hyperderivatives (see \cite[Lemma 2.4.1]{Pap} or \cite[\S 3.2]{CGM19}), a short calculation using hyperderivatives shows that  each $P_i^*[k',k]$ can be written in the following form
	\[ P_i^*[k',k]= \sum_{\substack{0 \leq i_\ell < \dots < i_{j-1}<i \\ c_\ell,\dots,c_{j-1},c \in \bN}} \frac{Q_{(c_\ell,\dots,c_{j-1},c)}}{\LL_{i_\ell}^{s_\ell+c_\ell} \dots \LL^{s_{j-1}+c_{j-1}}_{i_{j-1}} \LL_i^{s_j+\dots+s_r-s+c}} \bigg|_{t=\theta} \]
where $Q_{(c_\ell,\dots,c_{j-1},c)} \in \Fq[t,\theta]$ and $c_\ell+\dots+c_{j-1}+c < d_1$.

For $j \in \N$, we use the estimate
	\[ |L_j|_\nu=|\LL_j (\theta)|_\nu \geq |\nu|_\nu \]
which implies 
	\[ \lvert P_i^* \rvert_\nu \leq \lvert \nu \rvert_\nu^{-i(3d_1-1)}. \]
and thus
	\[ \lvert P_i^* \bv^{(i)} \rvert_\nu \leq \lvert \nu \rvert_\nu^{-i(3d_1-1)} \lvert \bv \rvert_\nu^{q^i}. \]
Since $|\bv|_\nu <1$, it follows that $\lvert P_i^* \bv^{(i)} \rvert_\nu$ tends to $0$ when $i \to +\infty$. This completes the proof.
\end{proof}

Recall that for $1 \leq \ell \leq r$, $d_\ell:=s_\ell+\dots+s_r$. We set 
	\[ a_\nu:=(\nu^{d_1}-1) \ldots (\nu^{d_r}-1). \]
The main result of this section is stated as follows.
\begin{theorem} \label{theorem: nu adic}
The series $\Log_{E'^*}(E'^*_{a_\nu} \bv^*_\fs)$ converges $\nu$-adically in $\Lie_{E'^*}(\C_\nu)$. Further, the $d_1$th coordinate of $\Log_{E'^*}(E'^*_{a_\nu} \bv^*_\fs)$ equals $-a_\nu \Gamma_{s_1} \ldots \Gamma_{s_r} \zeta_A(s_r,\dots,s_1)$.
\end{theorem}

\begin{remark}
Following Chang and Mishiba \cite[\S 6]{CM19b} we define $\zeta_A(s_r,\dots,s_1)_\nu$ to be the value $-\frac{1}{a}$ multiplied by the $d_1$th coordinate of $\Log_{E'^*}(E'^*_a \bv^*_\fs)_\nu$ for some nonzero element $a \in A$ with $|E'^*_a \bv^*_\fs|_\nu<1$. Note that this value does not depend on the choice of $a$ by \cite[Remark 6.2.5]{CM19b}. Hence Theorem \ref{theorem: nu adic} gives a logarithmic interpretation for $\zeta_A(s_r,\dots,s_1)_\nu$.

\end{remark}

\begin{proof}[Proof of Theorem \ref{theorem: nu adic}]
By Proposition \ref{prop: integral}, $\bv^*_\fs$ is a point in $E'^*(A)$. This implies that $|E'^*_{a_\nu} \bv^*_\fs|_\nu <1$ by \cite[Proposition 4.1.1 and Remark 4.2.4]{CM19}. Thus Theorem \ref{theorem: nu adic} follows immediately from Proposition \ref{proposition: nu adic convergence}.
\end{proof}

We give a brief application of the previous theorem.  The arguments given in \cite[\S 6.4]{CM19b}, which is based on Yu's sub-$t$-module theorem (see \cite[Theorem 0.1]{Yu97}), apply without any modification so that we obtain a proof of a conjecture of Furusho over function fields. The conjecture is stated as follows: if we denote by $\overline{\mathcal Z}_n$ (resp. $\overline{\mathcal Z}_{n,\nu}$) the $\overline K$-vector space generated by all $\infty$-adic (resp. $\nu$-adic) MZV's of weight $n$, the we have a well-defined surjective $\overline K$-linear map
	\[ \overline{\mathcal Z}_n \to \overline{\mathcal Z}_{n,\nu}, \quad \zeta_A(\fs) \mapsto \zeta_A(\fs)_\nu. \] 
We refer the reader to \cite{CM19b} for more details.

\subsection{Further remarks} ${}$\par


1) In \cite{Har19} Harada introduced alternating multiple zeta values in positive characteristic which are generalizations of Thakur multiple zeta values. Our machinery can apply easily to obtain logarithmic interpretations for alternating MZV's. 

2) In \cite{GND20} we investigated algebraic relations among Goss's zeta values for function fields of elliptic curves. As one crucial step of our analysis, we had to do some period calculations for some Anderson $t$-modules (see \cite[\S 3.3]{GND20}). Note that for Drinfeld modules these calculations follows immediately from basic properties of Anderson generating functions (see for example \cite[\S 4.2]{CP12}). Our method was based on direct calculations by taking advantage of working with elliptic curves. The motivation of this paper grows from our desire to generalize the aforementioned arguments for general curves. We expect that the method of this paper would provide a general approach to period calculations in our work in progress. 

\section{Relations with the works of Anderson-Thakur and Chang-Mishiba} \label{sec: examples}

This section is devoted to comparing the $t$-modules associated to the star dual $t$-motives defined in \S \ref{S:Log Int for MZVs} with those arising from the works of Anderson-Thakur \cite{AT90} and Chang-Mishiba \cite{CM19b}. We start with some examples given by Chang-Mishiba \cite{CM19b} and observe that, in these examples, for the same multiple zeta value, the $t$-module constructed by the star model has smaller dimension. Next we prove that indeed this inequality always holds. Finally, we determine integral points in special cases which covers all the examples given in \cite{AT90}.

\subsection{Setup} ${}$\par

In this section, let $\fs=(s_1,\dots,s_r) \in \N^r$ be a tuple with $r \geq 1$. For $1 \leq \ell \leq r$, we put $d_\ell=s_\ell+\dots+s_r$. In \cite{CM19b} Chang and Mishiba gave a logarithmic interpretation for $\zeta_A(\fs)$ (see \cite[Theorem 1.4.1]{CM19b}). More precisely, they constructed a $t$-module $G_\fs$ defined over $K$, a special point denoted by $\bv^{CM}_\fs \in G_\fs(K)$ and a vector $\bz^{CM}_\fs \in \Lie G_\fs(\bC_\infty)$ such that 

1) The $d_1$th coordinate of $\bz^{CM}_\fs \in G_\fs(K)$ equals $\Gamma_{s_1} \dots \Gamma_{s_r} \zeta_A(\fs)$.

2) $\exp_{G_s}(\bz^{CM}_\fs)=\bv^{CM}_\fs$.

We put $\fs'=\text{inv}(\fs)=(s_r,\dots,s_1)$. Then in \S \ref{S:Log Int for MZVs} we constructed a $t$-module $E'^*$ defined over $A$ and a special integral point $\bv^*_{\fs'} \in E'^*(A)$. Theorem \ref{theorem: star model} gives a split-logarithmic identity for $E'^*$:
\begin{align*}
\Log_{E'^*}^!(\bv^*_{\fs'})
=\delta_0 \begin{pmatrix}
- \mathfrak L(s_1,\dots,s_r) \Omega^{-(s_1+\dots+s_r)} \\
- \mathfrak L(s_1,\dots,s_{r-1}) \Omega^{-(s_1+\dots+s_{r-1})} \\
\vdots \\
- \mathfrak L(s_1) \Omega^{-s_1} 
\end{pmatrix}.
\end{align*}
In particular, the $d_1$th coordinate of the $\Log_{E'^*}^!(\bv^*_{\fs'})$ equals $- \Gamma_{s_1} \dots \Gamma_{s_r} \zeta_A(\fs)$.

In the depth-one case, i.e. when $r=1$ and $s=(n)$, both constructions coincide with that of Andersond and Thakur given in \cite{AT90}. The associated $t$-module is the $n$th tensor power $C^{\otimes n}$ of the Carlitz module. In \cite{AT90} they denoted by $\mathbf{Z}_n \in C^{\otimes n}(A)$ the special point and by $\bz_n \in \Lie C^{\otimes n}(\bC_\infty)$ the associated vector.

\subsection{Relation with the work of Chang-Mishiba} ${}$\par \label{sec: relation with CM}

We first give examples to compare the previous logarithmic interpretations for MZV's given by Chang-Mishiba \cite{CM19b} and by the star dual $t$-motives (see Theorem \ref{theorem: star model}). The two examples are taken from \cite{CM19b}. We observe that, in both cases, the Anderson $t$-module arising from the star model has smaller dimension and the associated integral point is ``simpler". 

\begin{example}
The following example is taken from \cite[Example 5.4.1]{CM19b}. We consider general $q$ and $\fs=(s_1=1,s_2=1,s_3=2)$. We have $\Gamma_1=\Gamma_2=1$ and $H_1=H_2=1$. 

On the one hand, the $t$-module $G_\fs$ has dimension $10$ and is given by
\begin{align*}
(G_\fs)_\theta=\left(
\begin{array}{c c c c | c | c  c | c c | c}
\theta & 1 &  &  &  &  &  &  &  &   \\
& \theta & 1 &  &  &  &  &  &  &   \\
&  & \theta & 1 &  &  &  &  &  &   \\
\tau &  &  & \theta & -\tau & -\tau &  &  -\tau &  & \tau   \\
\hline
&  &  &  & \theta + \tau &  &  &  &  &   \\
\hline
&  &  &  &  & \theta & \tau &  &  &   \\
&  &  &  &  & 1 & \theta & &  &   \\
\hline
&  &  &  &  &  &  & \theta & 1 &    \\
&  &  &  &  &  &  & \tau & \theta & -\tau   \\
\hline
&  &  &  &  &  &  &  &  & \theta + \tau
\end{array}
\right).
\end{align*}
Moreover, 
	\[ \bv^{CM}_\fs=(0,0,0,0,1,0,1,0,-1,1)^\top, \]
and 
	\[ \bz^{CM}_\fs=(*,*,*,\zeta_A(1,1,2),\text{Li}^*_1(1),*,\text{Li}^*_2(1),*,\text{Li}^*_{(1,1)}(1,1),\text{Li}^*_1(1))^\top. \]
	
On the other hand, the $t$-module $E'^*$ has dimension $7$ and is defined by
\begin{align*}
E'^*_\theta=\left(
\begin{array}{c c c c | c  c  | c }
\theta & 1 &  &  &  &  &    \\
& \theta & 1 &  &  &  &   \\
&  & \theta & 1 &  &  &    \\
\tau &  &  & \theta & \tau &  &    \\
\hline
&  &  &  & \theta & 1 &    \\
&  &  &  &  \tau & \theta & \tau   \\
\hline
&  &  &  &   &  & \theta 
\end{array}
\right).
\end{align*}
The special point is given by
	\[ -\bv^*_{\fs'}=(0,0,0,0,0,0,1)^\top. \]
By Theorem \ref{theorem: star model} we have
	\[ \Log_{E'^*}(-\bv^*_{\fs'})=(*,*,*,\zeta_A(1,1,2),*,\zeta_A(1,1),\zeta_A(1))^\top. \]
	
Note that 
	\[ \dim E'^*=7 < \dim G_\fs=10. \]
\end{example}

\begin{example}
The following example is taken from \cite[Example 5.4.2]{CM19b}. We take $q=2$ and $\fs=(s_1=1,s_2=3)$. We have $\Gamma_1=1,\Gamma_3=\theta^2+\theta$ and $H_1=1,H_3=t+\theta^2$. 

On the one hand, the $t$-module $G_\fs$ has dimension $6$ and is given by
\begin{align*}
(G_\fs)_\theta=\left(
\begin{array}{c c c c | c  | c }
\theta & 1 &  &  &  &   \\
& \theta & 1 &  &  &  \\
&  & \theta & 1 &  &   \\
\tau &  &  & \theta & -\theta^2 \tau & -\tau  \\
\hline
&  &  &  & \theta+\tau &   \\
\hline
&  &  &  &   & \theta+\tau
\end{array}
\right).
\end{align*}
Further
	\[ \bv^{CM}_\fs=(0,0,0,0,1,1,\theta+1)^\top, \]
and
	\[ \bz^{CM}_\fs=(*,*,*,(\theta^2+\theta) \zeta_A(1,3),\text{Li}^*_1(1),\theta \text{Li}^*_1(1))^\top. \]

On the other hand, the $t$-module $E'^*$ has dimension $5$ and is defined by
\begin{align*}
E'^*_\theta=\left(
\begin{array}{c c c c | c  }
\theta & 1 &  &  &  \\
& \theta & 1 &  &  \\
&  & \theta & 1 &  \tau \\
\tau &  &  & \theta & (\theta^2+\theta)\tau \\
\hline
&  &  &  & \theta 
\end{array}
\right)
\end{align*}
The special point is given by
	\[ -\bv^*_{\fs'}=(0,0,0,0,1)^\top. \]
By Theorem \ref{theorem: star model} we have
	\[ \Log_{E'^*}(-\bv^*_{\fs'})=(*,*,*,(\theta^2+\theta) \zeta_A(1,3),(\theta^2+\theta) \zeta_A(1))^\top. \]
	
Note that 
	\[ \dim E'^*=5 < \dim G_\fs=6. \]
\end{example}

The rest of this section aims to prove that an inequality of dimensions always holds.

\begin{proposition} \label{prop: dimension}
With the above notation, we have $\dim G_\fs \geq \dim E'^*$.

The equality holds if and only if either $r=1$ or $r=2$ and $\fs=(s_1,s_2)$ with $1 \leq s_1, s_2 \leq q$.
\end{proposition}

\begin{proof}
We set $w(\fs):=s_1+\ldots+s_r$ called the weight of $\fs$. By \cite[Theorem 5.2.5]{CM19b} there are explicit tuples $\fs_\ell \in \N^{\text{dep}(\fs_\ell)}$ with $w(\fs_\ell)=w(\fs)$, $\text{dep}(\fs_\ell) \leq r$, explicit coefficients $b_\ell \in A^*$ and vectors $\bu_\ell \in A^{\text{dep}(\fs_\ell)}$ so that	
	\[ \Gamma_{s_1} \ldots \Gamma_{s_r} \zeta_A(\fs) = \sum_\ell b_\ell \cdot (-1)^{\text{dep}(\fs_\ell)-1} \text{Li}^*_{\fs_\ell} (\bu_\ell). \]
	
Let $\fs_\ell$ be such a tuple. We write $\fs_\ell=(\fs_{\ell,1},\ldots,\fs_{\ell,\text{dep}(\fs_\ell)})$ and set
	\[ \epsilon(\fs_\ell):=(\fs_{\ell,1}+\ldots+\fs_{\ell,\text{dep}(\fs_\ell)-1})+\ldots+\fs_{\ell,1}. \]
Note that $\epsilon(\fs_\ell)$ belongs to $\N$. Then it is shown in \cite{CM19b} that
	\[ \dim G_\fs = (s_1+\ldots+s_r) + \sum_\ell \epsilon(\fs_\ell). \]

By the construction of the star model associated to $\zeta_A(\fs)$ we see that
	\[ \dim E'^* = (s_1+\ldots+s_r) + \epsilon(\fs). \]
The proposition follows from the fact that there exists $\ell_0$ such that $\fs_{\ell_0}=\fs$.

The equality holds if and only if $\fs_\ell = (s_1+\ldots+s_r)$ for $\ell \neq \ell_0$, which happens only when $r=1$ or $r=2$ and $\fs=(s_1,s_2)$ with $1 \leq s_1, s_2 \leq q$.
\end{proof}

\subsection{Relation with the work of Anderson-Thakur} ${}$\par \label{sec: relation with AT}

In this section we extend the previous examples to obtain the following result which explicitly computes integral points and covers all the examples given by Anderson-Thakur when $r=1$ (see \cite[page 187]{AT90}). 

By direct calculations we prove that for  $1 \leq n \leq q$, we have $H_n(t)=1$ and that for $q+1 \leq n \leq q^2$, we put $k=\lfloor \frac{n-1}{q} \rfloor$ and get
\begin{align*}
H_n(t) = \sum_{j=0}^k {n-jq+j-1 \choose j} (t^q-t)^{k-j} (t^q-\theta^q)^j.
\end{align*}
In particular, for $1 \leq n \leq q^2$, we always have
	\[ \deg H_n(t) \leq kq \leq n-1. \]

\begin{corollary}
Let $\fs=(s_1,\dots,s_r) \in \mathbb N^r$ be a tuple for $r \geq 1$ such that $1 \leq s_1 \leq q^2$. We denote by $\fs'=\text{inv}(\fs)=(s_r,\dots,s_1)$. If we express $H_{s_1}(t)=\sum_{i=0}^{s_1-1} a_i (t-\theta)^i$, then
	\[ -\bv^*_{\fs'}=(0,\dots,0,a_{s_1-1},\dots,a_1,a_0)^\top. \]
Further, this point belongs to the domain of convergence of $\Log_{E'^*}$.
\end{corollary}

\begin{proof}
We should keep in mind that we are working with the star model attached to $\fs'$. Since $Q_r:=H_{s_r'}=H_{s_1}=\sum_{i=0}^{s_1-1} a_i (t-\theta)^i$, Equation \eqref{eq: w_s} implies that
\begin{align*}
-\bv^*_{\fs'} &= \delta_1(\sigma Q_r m_r) \\
&= \delta_1(\sigma \sum_{i=0}^{s_1-1} a_i (t-\theta)^i m_r) \\
&=(0,\dots,0,a_{s_1-1},\dots,a_1,a_0)^\top.
\end{align*}
The proof is finished.
\end{proof}

\begin{remark}
1) When $r=1$, we recover the examples given by Anderson-Thakur when $r=1$ (see \cite[page 187]{AT90}). In this case, we take $r=1$ and $\fs=(n)$ with $1 \leq n \leq q^2$, hence $\fs'=\fs=(n)$. We see that the point $-\bv^*_{\fs'}$ coincides with the point $\mathbf{Z}_n$ defined by Anderson and Thakur. If we express $H_n(t)=\sum_{i=0}^{n-1} a_i (t-\theta)^i$, then
	\[ \mathbf{Z}_n=(a_{n-1},\dots,a_1,a_0)^\top. \]
	
2) Let $\fs=(s_1,\dots,s_r) \in \mathbb N^r$ be a tuple for $r \geq 1$ such that $1 \leq s_1 \leq q$. Thus $\fs'=\text{inv}(\fs)=(s_r,\dots,s_1)$. Since $H_{s_1}(t)=1$, we get  
\begin{align*}
-\bv^*_{\fs'}=(0,\dots,0,1)^\top= \begin{pmatrix}
0 \\
\vdots \\
0 \\
1 
\end{pmatrix}.
\end{align*}
	
3) Let $\fs=(s_1,\dots,s_r) \in \mathbb N^r$ be a tuple for $r \geq 1$ such that $q+1 \leq s_1 \leq 2q$. Thus $\fs'=\text{inv}(\fs)=(s_r,\dots,s_1)$. It follows that
\begin{align*}
H_{s_1}(t)=(t^q-t)+s_1(t^q-\theta^q)=(s_1+1)(t-\theta)^q-(t-\theta)+\theta^q-\theta.
\end{align*}
Then
\begin{align*}
-\bv^*_{\fs'} = \begin{pmatrix}
0 \\
\vdots \\
0 \\
s_1+1 \\
0 \\
\vdots \\
0 \\
-1 \\
\theta^q-\theta
\end{pmatrix}
\end{align*}
where $\theta^q-\theta$ is the $d$th coordinate and $s_1+1$ is the $(d-q)$th coordinate.
\end{remark}



\end{document}